\documentclass[reqno,10pt]{amsart}
 \usepackage[a4paper, left=30mm,right=30mm,top=30mm,bottom=30mm,marginpar=25mm]{geometry}
\usepackage{amsmath}
\usepackage{amsthm}
\usepackage{graphicx}
\usepackage{color}
\usepackage{enumerate}

\numberwithin{figure}{section}

\newcommand{\E}{\mathcal{E}}
\newcommand{\T}{\mathbb{T}}
\newcommand{\C}{\mathcal{C}}
\newcommand{\LL}{\mathcal{L}}

\DeclareMathOperator*{\argmin}{arg\,min}

\newtheorem{theorem}{Theorem}
\newtheorem{lemma}{Lemma}

\begin{document}
	

\title[Scheme for Maxwell-Stefan system]{An energy stable and positivity-preserving scheme for the Maxwell-Stefan diffusion system}
\author[X. Huo]{Xiaokai Huo}
\address{Computer, Electrical and Mathematical Science and Engineering Division,
King Abdullah University of Science and Technology (KAUST),Thuwal 23955, Saudi Arabia} \email{xiaokai.huo@kaust.edu.sa}
\author[H. Liu]{Hailiang Liu}
\address{Iowa State University, Mathematics Department, Ames, IA 50011 }
\email{hliu@iastate.edu}
\author[A.E. Tzavaras]{Athanasios E. Tzavaras}
\address{Computer, Electrical and Mathematical Science and Engineering Division,
King Abdullah University of Science and Technology (KAUST),Thuwal 23955, Saudi Arabia} \email{athanasios.tzavaras@kaust.edu.sa}
\author[S. Wang]{Shuaikun Wang}
\address{Computer, Electrical and Mathematical Science and Engineering Division,
King Abdullah University of Science and Technology (KAUST),Thuwal 23955, Saudi Arabia} \email{shuaikun.wang@kaust.edu.sa}

\begin{abstract}
We develop a new finite difference scheme for the Maxwell-Stefan diffusion system. The scheme is conservative, energy stable and positivity-preserving.  These nice properties stem from a variational structure and are proved by reformulating the finite difference scheme into an equivalent optimization problem. 
The solution to the scheme emerges as the minimizer of the optimization problem, and as a consequence  energy stability and positivity-preserving properties are obtained.
\end{abstract}


\maketitle

\section{Introduction}
Cross diffusion occurs in multicomponent systems, such as ionic liquids, wildlife populations, gas mixtures, tumor growth, etc \cite{jungel2016entropy,krishna1997maxwell}. In these multicomponent systems, the diffusion happens not only in the direction from high concentration to low concentration, but also in the opposite direction due to cross diffusion. In such cases, diffusion can not be described by Fick's diffusion law and the Maxwell-Stefan diffusion model can be
used instead. The Maxwell-Stefan model assumes the friction between two components is proportional to their difference in velocity and molecular fractions. It is widely used in modeling multicomponent systems.


In this work, we consider the Maxwell-Stefan diffusion system for a $n$-component mixture on the  torus $\mathbb{T}^d$, which reads for $i = 1, ... , n$,
\begin{align}
	\partial_t \rho_i +\nabla \cdot (\rho_i v_i) &= 0, \label{eq:f1}\\
	- \sum_{j=1}^n b_{ij} \rho_j (v_i-v_j) &= \nabla \log \rho_i - \frac{1}{\sum_{j=1}^n\rho_j} \sum_{j=1}^n \rho_j \nabla \log \rho_j , \label{eq:f2}\\
	\sum_{j=1}^n \rho_j v_j &= 0.\label{eq:consm}
\end{align}
Here $x\in\mathbb{T}^d$, $\rho_i=\rho_i(x,t)$ and $v_i=v_i(x,t)$ are the density and velocity of the $i$-th component.
The initial conditions are taken to be
\begin{align*}
	\rho_i(x,0) = \rho_{i0}(x), \; i= 1,\ldots,n,
\end{align*}
and we assume that 
\begin{equation}
\label{hypdata}
 \rho_{i0}(x) > 0 \, , \quad \mbox{and} \quad \sum_{j=1}^n \rho_{j0}(x) = 1 \quad \mbox{for $x \in \mathbb{T}^d$}.
\end{equation}
Solutions of \eqref{eq:f1} conserve the total mass 
$
	\partial_t \sum_{i=1}^n  \rho_i +\nabla \cdot \sum_{i=1}^n  \rho_i v_i = 0 \, .
$
Condition \eqref{eq:consm} imposes that the average velocity of the mixture is $v_{av} \equiv  0$ and thus the total
density $\sum_{i=1}^n \rho_i$ is conserved at each  $x \in \mathbb{T}^d$. Hypothesis \eqref{hypdata} then fixes
the total mass to
\begin{align}\label{eq:consrho}
		\sum_{j=1}^n \rho_j (x,t) = 1, \quad \text{ for } x\in\mathbb{T}^d, \; t>0  \, .
\end{align}
Accordingly,  \eqref{eq:f1}-\eqref{eq:consm} reduces to
\begin{align}
	&\partial_t \rho_i +\nabla \cdot (\rho_i v_i) = 0, \\
	&\nabla \rho_i  = - \sum_{j=1}^n b_{ij} \rho_i \rho_j(v_i-v_j), 
\end{align}
$i=1,\ldots,n$, which is the usual form of the Maxwell-Stefan diffusion system.

The system \eqref{eq:f1}-\eqref{eq:consm} can be obtained as the high-friction limit of the multicomponent Euler equations \cite{huo2019high}.
\begin{equation}
\label{eulerfr}
\begin{aligned}
	&\partial_t \rho_i + \nabla \cdot (\rho_i v_i) =0, \\
	&\partial_t (\rho_i v_i) + \nabla \cdot (\rho v_i v_i) + \frac{\rho_i}{\varepsilon}   \nabla \frac{\delta F(\rho)}{\delta \rho_i} = - \frac{1}{\varepsilon} \sum_{j=1}^n b_{ij} \rho_i\rho_j(v_i-v_j),
\end{aligned}
\end{equation}
when the total momentum (or the mean velocity) is zero.
Here the energy functional
\begin{align}\label{eq:fde}
	F(\rho) = \sum_{i=1}^n \int_{\mathbb{T}^d} \rho_i(x) \log \rho_i(x) dx.	
\end{align}
It was proved in \cite{huo2019high} that, when the total momentum is zero, the system \eqref{eulerfr} converges to \eqref{eq:f1}-\eqref{eq:consm} 
in the high-friction limit  $\varepsilon \to 0$. 
Moreover, \eqref{eq:f1}-\eqref{eq:consm} can be regarded as a gradient flow for $F(\rho)$.

This raises the following question: Given densities $\rho^0 = (\rho_{i}^0)_{i=1}^n$,
$\rho^1 = (\rho_{i}^1)_{i=1}^n$, with $\sum_i \rho_{i}^0 = \sum_i \rho_{i}^1= 1$, consider the minimization problem
\begin{align}
	 \min_{ (\rho,v) \in K } \int_0^1  \! \!  \int_{\mathbb{T}^d} \sum_{i,j=1}^n  \frac{1}{4} b_{ij} \rho_i\rho_j(v_i-v_j)^2 dxdt 
\label{blockmini}
\end{align}
over the set
\begin{align*}
K = \bigg \{ \rho = (\rho_1, ... , \rho_n) \, , \;  &v = (v_1 , ... , v_n) \; : \;  \partial_t \rho_i + \nabla \cdot (\rho_i v_i)=0, ~~i=1,\ldots,n,
\\
	&\sum_{j=1}^n \rho_j  v_j=0 \, , \; \;  \rho_i(0,x)=\rho_i^0(x) \, , \; \rho_i(1,x)=\rho_i^1(x)   \bigg \}.
\end{align*}
The problem \eqref{blockmini} as the minimum of the frictional work is motivated by the well-known characterization of the Wasserstein distance  in a one-component fluid obtained by Benamou-Brenier  \cite{benamou2000computational}. 
The study of this question will be given in a forthcoming work.  The minimization \eqref{blockmini} and 
the gradient structure of \eqref{eq:f1}-\eqref{eq:consm} detailed in \cite{huo2019high}, 
motivate us to use the work of friction as a building block for a numerical scheme of variational provenance -- in the spirit of the well known JKO scheme \cite{jko98} -- in order to exploit the gradient structure of the Maxwell-Stefan system. This connection is pursued in the present work.

In this paper, we develop a new implicit-explicit finite difference scheme for the Maxwell-Stefan system \eqref{eq:f1}-\eqref{eq:consm} and prove that the scheme is energy dissipating and positivity preserving, for arbitrary time step and spatial meshes. The scheme in one dimension takes the form:
\begin{align}
	\frac{\rho_{i}^{k+1}-\rho_{i}^k}{\Delta t} + d_h (\hat{\rho}_i^k v_i^{k+1}) &=0,\label{eq:sc1}
	\\
	 - \sum_{j=1}^n b_{ij}  \hat{\rho}_{j}^k (v_{i}^{k+1}-v_{j}^{k+1}) &=
	D_h \log \rho_i^{k+1} - \frac{1}{\sum_{j=1}^n \hat{\rho}_{j}^k}\sum_{j=1}^n \hat{\rho}_{j}^k 
	 D_h \log \rho_j^{k+1},
	\label{eq:sc2}\\
	\sum_{j=1}^n \hat{\rho}_{j}^k v_{j}^{k+1} &= 0
	\label{eq:sc3}
\end{align}
(for the $d$-dimensional case the reader is referred to Section  \ref{sec:mu}). 
The subscript $i$ refers to the $i$-th component and takes values $i = 1, ... , n$, while the superscript $k$ refers to the $k$-th time step.  
 The equations \eqref{eq:sc1}-\eqref{eq:sc3} are computed at spatial grid points $\ell$ or $\ell+\tfrac{1}{2}$ of staggered lattices in a way precised in Section \ref{sec:sc}.
The parameter $\Delta t$ is the time step and $h$ is the mesh size.  The operators $d_h$, $D_h$ are central difference operators, in one dimension, defined by 
\begin{align}\label{eq:Dhdh}
	(d_h f_{i})_\ell = \frac{f_{i,\ell+1/2}-f_{i,\ell-1/2}}{h} ,\quad (D_h f_{i})_{\ell+\frac{1}{2}} = \frac{f_{i,\ell+1}-f_{i,\ell}}{h},
\end{align}
where $\ell=\{1,\ldots,N\}$, $N$  the number of mesh intervals, and we set
	$(\hat{f}_i)_{\ell+\frac12} = \frac12(f_{i,\ell} + f_{i,\ell+1}).$

The scheme is induced by a spatial discretization of the constrained optimization problem ({\it cf.} \eqref{eq:op})
\begin{align}\label{varmin}
	\min_{K} \bigg \{  \int_{\mathbb{T}^d} \Delta t  \sum_{i,j=1}^n \frac{1}{4} b_{ij} \rho_i^k\rho_j^k |u_i-u_j|^2 dx + \int_{\mathbb{T}^d} \sum_{j=1}^n \rho_j \log \rho_j \, dx \bigg \},
\end{align}
where the set $K$ is defined to be 
\begin{align*}
	K=\left\{(\rho,v):~\rho>0,~\frac{\rho_i - \rho_i^k }{\Delta t} + \nabla \cdot (\rho_i^k u_i) = 0,~\sum_{i=1}^n \rho_i^k  u_i =0\right\}.
\end{align*}
The approach is motivated by the JKO-scheme \cite{jko98} and  the Benamou-Brenier interpretation of the Wasserstein distance \cite{benamou2000computational}, the latter
 suggesting
an alternate variational scheme for nonlinear Fokker-Planck equations espoused in \cite{liluwang2020}. The novelty here is (i)  that the limiting problem is a coupled parabolic system
and (ii) that the mechanical friction is a complex interaction among the different components (see \cite{bothe2011maxwell}) that is only captured in bulk by the dissipation  functional \eqref{blockmini}.
Nevertheless, this suffices in capturing the detailed interaction.

We show that there exists a discrete energy function which dissipates along time iterations,  and that the numerical solutions for the densities generated by
the scheme \eqref{eq:sc1}-\eqref{eq:sc3}  preserve the positivity of the initial densities.
The proof uses variational arguments and is based on the reformulation of the finite difference scheme  as an equivalent optimization problem.
An interesting feature is the role played by an elliptic operator $\mathcal{L}_\Phi$ defined in \eqref{ellipticop} and the induced dual norm \eqref{eq:linversenorm}.
The reader familiar with the Wasserstein distance will recognize analogies with duality induced norms \cite{otto01,ottowest06,ottotzavaras08} appearing in the theory of
nonlinear Fokker-Planck equations and induced by the metric tensor generating the Wassertein metric.

A large literature 
\cite{bothe2011maxwell,boudin2012mathematical,geiser2015numerical,giovangigli1998local,jungel2016entropy,jungel2019convergence,jungel2013existence}
employing diverse techniques has provided a basic theory for the Maxwell-Stefan system  \eqref{eq:f1}-\eqref{eq:consm}. 
The existence of global weak solutions is established in \cite{jungel2013existence}, while local existence of strong solutions was shown in \cite{bothe2011maxwell,giovangigli1998local}. 
Explicit finite difference schemes were developed in \cite{boudin2012mathematical,geiser2015numerical}. An implicit Euler Galerkin scheme was developed in
 \cite{jungel2019convergence} for the Maxwell-Stefan system coupled with a Poisson equation. The scheme was also shown to satisfy a discrete entropy inequality.
 However, the property of preserving positivity has not been investigated in the above works. 
 The present work provides a connection between finite difference schemes and variational minimization problems. This approach  is quite robust and we expect that, once the theory for the continuous problem \eqref{varmin} is further developed, it will lead to theoretical results for more complicated
 schemes such as finite elements.

Recently there has been a growing interest in developing energy stable and/or positivity-preserving numerical schemes for nonlinear diffusion equations \cite{chen2019positivity,dong2019positivity,gong2020arbitrarily,huo2019positivity,liu2019positive,liu2020efficient,shen2019new}.  
Positivity-preserving schemes for the Poisson-Nernst-Planck systems were developed in \cite{liu2019positive,liu2020efficient}, where the maximum principle was used to show the non-negativity of the scheme. A series of diffusion equations satisfying a gradient flow structure was considered in \cite{chen2019positivity,dong2019positivity,gong2020arbitrarily,shen2019new},  where energy-stabie schemes were developed for the Cahn-Hillard equations, with positivity-preserving properties proved in \cite{chen2019positivity,dong2019positivity} via optimization formulations. The technique was also used in \cite{huo2019positivity} to prove the positivity and energy-stability properties for a scheme associated to the quantum diffusion equation. Our approach  extends such works to a setting of  systems that are gradient flows
by exploiting the frictional dissipation natural to the Maxwell-Stefan system.

The structure of the paper is as follows: in Section \ref{sec:sc}, we give the details of the numerical scheme and show that it conserves the total mass and is consistent. In Section \ref{sec:op}, we first prove that the numerical scheme is equivalent to an optimization problem, in Theorem \ref{theorem}, and then show the energy stability and positivity-preserving properties in Theorem \ref{thm:prop}.
 We provide the multidimensional scheme in Section \ref{sec:mu} and show that similar properties also hold. Finally, we give some numerical examples to verify the proved properties.

\section{The scheme}\label{sec:sc}
\subsection{Notations} \label{sec:not}
We use notations from \cite{wise2009energy}.
We define the following two grids on the torus $\T=[0,L]$ with mesh size $h=L/N$, where $N$ is the number of mesh intervals:
\begin{align}\label{eq:cedef}
\C:=\{h,2h,\ldots,L\}, \quad \E:=\left\{\frac{h}{2},\frac{3h}{2},\ldots,(N- \tfrac{1}{2})h 
\right\}.
\end{align}
We define the discrete $N$-periodic function spaces as
\[\C_{\rm per}:=\{f:\C \to \mathbb{R}\},\quad \E_{\rm per}:=\{f:\E\to\mathbb{R}\}.\]
Here we call $\C_{\rm per}$ the space of \emph{cell centered functions} and $\E_{\rm per}$ the space of \emph{edge centered functions}.  We use $f_{\ell}$ to denote the value of function $f$ at grid point $x_{\ell}=\ell h$. We also define the subspace
$\mathring{\C}_{\rm per}:=\left\{f: f\in \C_{\rm per},\,\sum_{\ell=1}^N f_\ell = 0\right\}.$
We can extend the above definitions to vector value functions. For example, we define $\C_{\rm per}^n$ by
\begin{align*}
	\C_{\rm per}^n:=\{f=(f_1,\ldots,f_n) : f_i \in \C_{\rm per},~~i=1,\ldots,n\}.
\end{align*}
The spaces $\E_{\rm per}^n$, $\mathring{\C}_{\rm per}^n$ are defined the same way.
The discrete gradients $D_h$ and $d_h$ are defined in \eqref{eq:Dhdh}. 
We define the average of the function values of nearby points by
\begin{align}\label{eq:Axax}
\hat{f}_{\ell+\frac{1}{2}} = \frac{f_\ell + f_{\ell+1}}{2}, \text{ if } f \in \mathcal{C}_{\rm per}, \quad\text{and}\quad \hat{f}_{\ell} = \frac{f_{\ell+\frac{1}{2}}+f_{\ell-\frac{1}{2}}}{2}, \text{ if } f \in \mathcal{E}_{\rm per}.
\end{align}
The inner products 
are defined by
    $\langle f,g \rangle := h \sum_{\ell=1}^N f_\ell g_\ell, ~~\forall f,g \in \C_{\rm per}, \text{ and }[f,g]:=h\sum_{\ell=1}^N f_{\ell+\frac12}g_{\ell+\frac12},~~    
\forall f,g \in \E_{\rm per}.$
They can be also extended on $\C_{\rm per}^n$ and $\E_{\rm per}^n$ with
\begin{align*}
	\langle f,g \rangle:=h\sum_{i=1}^n\sum_{\ell=1}^Nf_{i,\ell}g_{i,\ell},~~\forall f,g\in \C_{\rm per}^n,\quad [f,g]:=h\sum_{i=1}^n\sum_{\ell=1}^N f_{i,\ell+\frac12} g_{i,\ell+\frac12}.
\end{align*}
We also take the following notation:
\begin{align*}
	\langle f \rangle :=h\sum_{\ell=1}^N f_\ell,~~f\in \C_{\rm per},\quad [f]:=h\sum_{\ell=1}^N f_{\ell+\frac12}, ~~ f\in \E_{\rm per}.
\end{align*}
Suppose $f\in \C_{\rm per}$ and $\phi \in \E_{\rm per}$, the following summation-by-parts formula holds:
\begin{align}
  \langle f,d_h \phi \rangle = -[D_h f,\phi].
\end{align} 
Next, we introduce a norm on $\mathring{\C}_{\rm per}^{n-1}$.
	 Let $\Phi$ be a $(n-1)\times (n-1)$ symmetric, positive definite matrix, with $\Phi_{ij}\in \E_{\rm per}$, $i,j=1,\ldots,n-1$.
	 We introduce the operator $\mathcal{L}_{\Phi}$ on $\mathring{\C}_{\rm per}^{n-1}$ defined by
	\begin{align}
		\mathcal{L}_{\Phi}  f := -d_h(\Phi D_h f) = \left(- \sum_{j=1}^{n-1} d_h(\Phi_{ij} D_hf_{j} )\right),~~\forall f \in \mathring{\C}_{\rm per}^{n}.
        \label{ellipticop}
	\end{align}
    For any $g\in \mathring{\C}_{\rm per}^{n-1}$, let $f$ be determined by $g = \mathcal{L}_{\Phi}f$, we define the norm
    \begin{align}\label{eq:linversenorm}
		\|g\|_{\mathcal{L}^{-1}_{\Phi}}^2: = [ D_h  f,  \Phi D_h f].
    \end{align}

\subsection{The scheme}
The scheme \eqref{eq:sc1}-\eqref{eq:sc3} is written in the component form as follows:
\begin{align}
	&\frac{\rho_{i,\ell}^{k+1}-\rho_{i,\ell}^k}{\Delta t} 
	= -\frac{1}{h}\left(\hat{\rho}_{i,\ell+\frac{1}{2}}^k v_{i,\ell+\frac{1}{2}}^{k+1} - \hat{\rho}_{i,\ell-\frac{1}{2}}^k v_{i,\ell-\frac{1}{2}}^{k+1} \right) , \label{eq:sd1} 
\\
	&-\sum_{j=1}^n b_{ij} \hat{\rho}_{j,\ell+\frac{1}{2}}^k (v_{i,\ell+\frac{1}{2}}^{k+1}-v_{j,\ell+\frac{1}{2}}^{k+1}) \label{eq:sd2}
\\
 &\qquad= \frac{\log \rho_{i,\ell+1}^{k+1} - \log \rho_{i,\ell}^{k+1}}{h}
 -  \frac{1}{h\sum_{j=1}^n \hat{\rho}_{j,\ell+\frac12}^k} 
	\sum_{j=1}^n \hat{\rho}^k_{j,\ell+\frac{1}{2}} (\log \rho_{j,\ell+1}^{k+1} - \log \rho_{j,\ell}^{k+1}),\nonumber
\\
	&\sum_{j=1}^n \hat{\rho}_{j,\ell+\frac12}^k v_{j,\ell+\frac12}^{k+1}=0,  \label{eq:sd3}
\end{align}
subject to initial data 
\begin{align}\label{ini}
\rho_{i, \ell}^0=\rho_{i0}(x_\ell), \quad i=1,\ldots, n, \quad \ell=1,\ldots, N.
\end{align}
 The scheme \eqref{eq:sd1}-\eqref{eq:sd3} is an implicit-explicit finite difference scheme. It can be obtained formally by discretizing the system \eqref{eq:f1}-\eqref{eq:consm}.

Next we study the conservation properties of the scheme. First we show that, at each grid point, the total mass is preserved.
\begin{lemma}\label{lm1}
	Suppose the solutions to the scheme \eqref{eq:sc1}-\eqref{eq:sc3} are positive for $k\ge 1$. Then the total mass at each grid point is conserved, i.e.
	\begin{align}\label{eq:lm1}
	 	\sum_{i=1}^n \rho_{i,\ell}^k = \sum_{i=1}^n \rho_{i,\ell}^0, \qquad \mbox{$\ell=1,\ldots,N$ and $k\ge1$}.
	 \end{align} 
\end{lemma}
\begin{proof}
	From equations \eqref{eq:sd1} and \eqref{eq:sd3}, we have for $\ell=1,\ldots,N$,
	\begin{align*}
		\sum_{i=1}^n \rho_{i,\ell}^{k+1} =& \sum_{i=1}^n \rho_{i,\ell}^k - \Delta t  \sum_{i=1}^n d_h(\hat{\rho}_i^k v_i^{k+1})_{\ell} \\
		=& \sum_{i=1}^n \rho_{i,\ell}^k - \frac{\Delta t}{h} \left(\sum_{i=1}^n \hat{\rho}_{i,\ell+\frac{1}{2}}^k v_{i,\ell+\frac{1}{2}}^{k+1}-\sum_{i=1}^n \hat{\rho}_{i,\ell-\frac{1}{2}}^k v_{i,\ell-\frac{1}{2}}^{k+1}\right) = \sum_{i=1}^n \rho_{i,\ell}^k.
	\end{align*}
	We take $k$ iteratively to get \eqref{eq:lm1}.
\end{proof}

Next, we show that for each component, the mass is conserved, i.e. the summation over grid points is conserved. The following lemma holds.
\begin{lemma}\label{lm2}
	Suppose the solutions to the scheme \eqref{eq:sc1}-\eqref{eq:sc3} are positive for any $k\ge 1$. Then the mass for each component is conserved, i.e.,
	\begin{align}\label{eq:lm2}
		\sum_{\ell=1}^N \rho_{i,\ell}^k = \sum_{\ell=1}^N \rho_{i,\ell}^0, \qquad \mbox{$i = 1,\ldots,n$, $k\ge 1$} \, .
	\end{align}
\end{lemma}
\begin{proof}
	From \eqref{eq:sd1}, we get
	\begin{align*}
		\sum_{\ell=1}^N \rho_{i,\ell}^{k+1} =& \sum_{\ell=1}^N \rho_{i,\ell}^{k} - \frac{\Delta t}{h} \sum_{\ell=1}^N \left(\hat{\rho}_{i,\ell+\frac{1}{2}}^k v_{i,\ell+\frac{1}{2}}^{k+1} - \hat{\rho}_{i,\ell-\frac{1}{2}}^k v_{i,\ell-\frac{1}{2}}^{k+1} \right) 
		=&\sum_{\ell=1}^N \rho_{i,\ell}^{k} .
	\end{align*}
	Iterating in $k$ we obtain \eqref{eq:lm2}.
\end{proof}

\subsection{The scheme in $n-1$ components}
We consider first the solvability of the algebraic system \eqref{eq:f2}-\eqref{eq:consm} under the hypothesis $b_{ij} > 0$.
Since summing the equations \eqref{eq:f2} in $i=1,\ldots,n$ equals zero, these $n$ equations are not independent.  
One easily checks that for $\rho_i > 0$ the homogeneous system 
$$
- \sum_{j=1}^n b_{ij} \rho_j (v_i-v_j) = 0
$$
has only the trivial solution $v_1 = \cdots =v_n$. Hence the null space has dimension one. The solution of \eqref{eq:f2}-\eqref{eq:consm} is given by the following lemma.

\begin{lemma}
	Let $\rho_i(x,t)>0$, $x\in\mathbb{T}^d,t>0$, $i=1,\ldots,n$,
	 and suppose that $b_{ij} > 0$  and $b_{ij}=b_{ji}$, for $i\neq j$ and $ i, j = 1, ... , n$. 
	Then the algebraic system \eqref{eq:f2},  \eqref{eq:consm} has a unique solution that is explicitly expressed by
	\begin{align*}
		&\rho_i v_i = -\sum_{j=1}^{n-1} D_{ij} \nabla (\log \rho_j - \log \rho_n),~ i=1,\ldots,n-1, ~~\text{and}~~\rho_n v_n =  -\sum_{i=1}^{n-1} \rho_i v_i,
	\end{align*}
	where
	\begin{align}
		D_{ij} = D_{ij}(\rho)= \sum_{s,m=1}^{n-1} Q^{-T}_{is} B^{-1}_{sm}Q^{-1}_{mj}, ~~ i,j = 1,\ldots,n-1,\label{eq:Dmatrix} \, 
	\end{align}
	and
	\begin{align}
	B_{ij} = B_{ij}(\rho) =& \delta_{ij}\sum_{m=1}^n b_{im} \rho_i\rho_m - b_{ij} \rho_i\rho_j,\label{eq:taumatrix} \\
		Q_{ij} = Q_{ij}(\rho) &= 
		\frac{1}{\rho_i}\delta_{ij} + \frac{1}{\rho_n}
		\label{eq:Qmatrix}
		\\
		(Q^{-1})_{ij} =  Q^{-1}_{ij} (\rho) &=  \delta_{ij} \rho_i - \frac{\rho_i \rho_j}{\sum_{j=1}^n \rho_j} .
		\label{eq:Qinvmatrix}
	\end{align}
\end{lemma}
For $\rho>0$, $B$ is diagonally dominant and thus invertible. We note that  $Q^T = Q$ and that by a direct computation $Q Q^{-1} = Q^{-1} Q = \mathbb{I}$, where  $Q^{-1}$ is determined by \eqref{eq:Qinvmatrix}; 
 hence, $Q$ is also invertible.
 The proof can be found in \cite{huo2019high} or \cite{yang2015rigorous}.
A similar formula is established for the numerical scheme \eqref{eq:sc1}-\eqref{eq:sc3}:
\begin{lemma}\label{lm4}
	Assume $b_{ij}>0$ and $b_{ij}=b_{ji}$ for $i\neq j$ and $i,j=1,\ldots,n$. Suppose ${\rho}^k_{i,\ell}>0$ for $i=1,\ldots,n$, $\ell=1,\ldots,N$. The solutions of \eqref{eq:sc2}-\eqref{eq:sc3} are calculated 
	by the explicit formula
	\begin{align}
		&\hat{\rho}_i^k v_i^{k+1} = - \sum_{j=1}^{n-1} \hat{D}_{ij}^k
		D_h(\log \rho_j^{k+1} - \log \rho_n^{k+1}), 
		~i=1,\ldots,n-1,
		\label{eq:lm4}%
	\end{align}
	and $\hat{\rho}_n^k v_n^{k+1} = -\sum_{i=1}^{n-1} \hat{\rho}_i^k v_i^{k+1}$.
	Here
	\begin{align}\label{eq:Dhat}
		\hat{D}_{ij}^k = \sum_{s,m=1}^{n-1}(\hat{Q}^k)^{-T}_{is} (\hat{B}^k)^{-1}_{sm} (\hat{Q}^k)^{-1}_{mj} \, ,
	\end{align}
	and $\hat{Q}_{ij}^k=Q_{ij} (\hat{\rho}^k)$, $\hat{B}_{ij}^k=B_{ij}(\hat{\rho}^k)$,  $(\hat{Q}^k)^{-1}_{ij} =Q^{-1}_{ij}(\hat{\rho}^k)$  are the corresponding matrices \eqref{eq:taumatrix}-\eqref{eq:Qinvmatrix}\textbf{} with $\rho_i$ replaced by $\hat{\rho}_i^k$.
\end{lemma}
Notice that formulas \eqref{eq:lm4}
 hold at each grid point $\ell+1/2=3/2,\ldots,N/2+1(\text{or } 1/2)$; to simplify the notation, we do not write the subscript $\ell+1/2$.
\begin{proof}
	Multiplying \eqref{eq:sc2} by $\hat{\rho}_i^k$ gives
\begin{align*}
	\hat{\rho}_i^k D_h\log\rho_i^{k+1} - \frac{\hat{\rho}_i^k}{\sum_{j=1}^n \hat{\rho}_j^k} \sum_{j=1}^n \hat{\rho}_j^k D_h \log \rho_j^{k+1} = -\sum_{j=1}^n b_{ij} \hat{\rho}_i^k \hat{\rho}_j^k (v_i^{k+1}-v_j^{k+1}) \, , 
\end{align*}
which is rewritten as
\begin{align}\label{eq:tm1}
	\sum_{j=1}^n\left(\delta_{ij} \hat{\rho}_i^k - \frac{\hat{\rho}_i^k \hat{\rho}_j^k}{\sum_{j=1}^n \hat{\rho}_j^k}\right)D_h\log\rho_j^{k+1} = -\sum_{j=1}^n \left(\delta_{ij} \sum_{m=1}^n b_{im}\hat\rho_i^k\hat\rho_m^k - b_{ij} \hat\rho_i^k \hat\rho_j^k  \right)v_j^{k+1}.
\end{align}
Setting
	$\hat{B}^k_{ij} = B_{ij}(\hat{\rho}^k)=  \delta_{ij} \sum_{m=1}^n b_{im}\hat\rho_i^k\hat\rho_m^k - b_{ij} \hat\rho_i^k \hat\rho_j^k,$
the right side of \eqref{eq:tm1} is expressed as
\begin{align}\label{eq:tm11}
	-\sum_{j=1}^n \hat{B}^k_{ij} v_j^{k+1} = -\sum_{j=1}^{n-1} \hat{B}_{ij}^k v_j^{k+1} - \hat{B}_{in}^kv_n^{k+1} = -\sum_{j=1}^{n-1} \hat{B}_{ij}^k(v_j^{k+1} -v_n^{k+1} ).
\end{align}
Using \eqref{eq:sc3}
we get
\begin{align}
	 -\sum_{j=1}^{n-1} &\hat B^k_{ij}(v_j^{k+1} -v_n^{k+1} )=-\sum_{j=1}^{n-1}\hat B^k_{ij}(v_j^{k+1} + \frac{1}{\hat\rho_n^k}\sum_{m=1}^{n-1}\hat\rho_m^k v_m^{k+1})\nonumber\\
	 & = -\sum_{j=1}^{n-1}\hat B^k_{ij}  \sum_{m=1}^{n-1}  (\frac{1}{\hat\rho_m^k}\delta_{jm} + \frac{1}{\hat\rho^k_n}) \hat\rho_m^k v_m^{k+1} = -\sum_{j,m=1}^{n-1} \hat B^k_{ij} \hat Q_{jm}^k \hat\rho_m^k v_m^{k+1} \, ,   
	 \label{eq:asd}
\end{align}
where 
$\hat Q_{jm}^k=Q_{jm}(\hat{\rho}^k)=\frac{1}{\hat\rho_m^k}\delta_{jm} + \frac{1}{\hat\rho^k_n}.$
 By direct calculation it is shown that   $\hat Q_{jm}^k$ is invertible with inverse
    $(\hat Q^k)^{-1}_{ij} = \left(\delta_{ij} \hat\rho_i^k - \frac{\hat\rho_i^k \hat\rho_j^k}{\sum_{j=1}^n \hat\rho_j^k}\right).$
The left side of \eqref{eq:tm1} is rewritten for $i \ne n$ as
\begin{align*}
	&\sum_{j=1}^n\left(\delta_{ij} \hat\rho_i^k - \frac{\hat\rho_i^k \hat\rho_j^k}{\sum_{j=1}^n \hat\rho_j^k}\right)D_h \log \rho_j^{k+1} \\
	&\quad = \sum_{j=1}^{n-1} (\hat Q^k)^{-1}_{ij} D_h \log \rho_j^{k+1} -\frac{\hat\rho_i^k(\sum_{j=1}^n\hat\rho_j^k -\sum_{j=1}^{n-1}\hat\rho_j^k )}{\sum_{j=1}^n \hat\rho_j^k}D_h \log \rho_n^{k+1} \\
	&\quad = \sum_{j=1}^{n-1} (\hat Q^k)^{-1}_{ij} D_h(\log \rho_j^{k+1} - \log \rho_n^{k+1}).
\end{align*}
This leads to expressing \eqref{eq:tm1} as
$$
\sum_{j=1}^{n-1} (\hat Q^k)^{-1}_{ij} D_h(\log \rho_j^{k+1} - \log \rho_n^{k+1}) =  -\sum_{j,m=1}^{n-1} \hat B^k_{ij} \hat Q_{jm}^k \hat\rho_m^k v_m^{k+1} \, .
$$
Since $\hat B^k$ and $\hat Q^k = (\hat {Q^k})^T$ are invertible, we conclude that \eqref{eq:lm4} holds.
\end{proof}

We adopt the notation 
\begin{align}\label{eq:tildenotation}
	\tilde{f} = (f_1,\ldots,f_{n-1}) \text{ for } f=(f_1,\ldots,f_n).
\end{align}
With Lemma \ref{lm4}, the scheme \eqref{eq:sc1}-\eqref{eq:sc3} can be written as
\begin{align*}
	\frac{\tilde{\rho}^{k+1}-\tilde{\rho}^k}{\Delta t} = -d_h\left(\hat{D}^k D_h \left(\frac{1}{h}\frac{\partial F_h}{\partial \tilde\rho}(\tilde\rho^{k+1})\right)\right),
\end{align*}
where
\begin{align}\label{lem:7}
    F_h= F_h(\tilde{\rho}):=\left\langle \sum_{i=1}^{n-1} \rho_i\log\rho_i\right\rangle 
     + \left\langle\left(1 - \sum_{i=1}^{n-1}\rho_i\right)\log \left(1 - \sum_{i=1}^{n-1}\rho_i\right) \right\rangle.
\end{align}


\subsection{Consistency}
Let $(P,V)$ be the exact smooth solution of the equations \eqref{eq:f1}-\eqref{eq:f2} in the space
    $P, V \in C^3_{t,x}([0,T]\times\mathbb{T}).$
The values at grid points are $P_{i,\ell}^k:=P_i(x_\ell,k\Delta t), V_{i,\ell}^k :=V_i(x_\ell,k\Delta t)$. The local truncation errors are defined by
\begin{align*}
    &\tau_{i}^1 = \frac{P_i^{k+1}-P_i^k}{\Delta t} +d_h(\hat{P}_i^k V_i^{k+1}), \\
	&\tau_i^2 = D_h \log P_i^{k+1} - \frac{1}{\sum_{j=1}^n \hat{P}_j^k}\sum_{i=1}^n \hat{P}_i^k D_h \log P_i^{k+1} +\sum_{j=1}^n b_{ij}  \hat{P}_j^k (V_i^{k+1}-V_j^{k+1}),\\
	&\tau_i^3=\sum_{i=1}^n \hat{P}_i^k V_i^{k+1}.
\end{align*}
We have the following lemma.
\begin{lemma}\label{lema}
 Suppose the solutions $(P,V)$ to the system \eqref{eq:f1}-\eqref{eq:consm} are smooth in time and space, with $P,V \in C_{t,x}^3$ and $P_i(x,t)>0$ for $x\in\mathbb{T}$ and $t>0$ and for any $i=1,\ldots,n$. Suppose $(P,V)$ satisfies the condition \eqref{hypdata}.
     Then the local truncation errors satisfy
        \begin{align*}
            |\tau_{i,\ell}^1|,~|\tau_{i,\ell+\frac12}^2|,~|\tau_{i,\ell+\frac12}^3| \le C(\Delta t + h^2).
        \end{align*}
        Here $C>0$ is a positive constant depending on $(P,V)$.
 \end{lemma}
 An elementary verification is deferred to Appendix A. 

\section{Optimization formulation}\label{sec:op}
\subsection{Formulation via an optimization problemma}
In this section, we give an optimization formulation of the scheme \eqref{eq:sc1}-\eqref{eq:sc3}. We recall that the system \eqref{eq:f1}-\eqref{eq:consm} can be written as the gradient flow of the energy functional \eqref{eq:fde}, see \cite{huo2019high}. Consider the minimization problem
\begin{align*}
	\rho^{k+1} = \argmin_{\rho\geq 0, w}\left\{\frac{1}{\Delta t} \int_{\mathbb{T}^d} \sum_{i,j=1}^n \frac{1}{4} b_{ij} \rho_i^k\rho_j^k(w_i-w_j)^2 dx + F(\rho)\right\},
\end{align*}
with $F(\rho)$ defined in \eqref{eq:fde},
subject to the constraints
\begin{align*}
	\rho_i - \rho_i^k + \nabla \cdot (\rho_i^k w_i) = 0, ~~ i =1,\ldots,n,~~\text{and}~~
	\sum_{i=1}^n \rho_i^k  w_i =0 .
\end{align*}
The idea is to calculate minimizers of the free energy penalized by the work  consumed by friction. The variational scheme is related
to the Jordan-Kinderlehrer-Otto scheme \cite{jko98}, an analogy due to the connection between frictional dissipation and the Wasserstein distance offered by
the Benamou-Brenier interpretation \cite{benamou2000computational} of the Monge-Kantorovich mass transfer problem. There is however one important difference, as
the frictional dissipation is more elaborate in the multi-component mixture situation.

The minimizers of the above constraint problem can be calculated by considering the min-max augmented Lagrangian
\begin{align*}
	\min_{\rho,w}\max_{\alpha,\beta} L(\rho, w, \alpha, \beta)  =& \frac{1}{\Delta t} \int_{\mathbb{T}^d} \sum_{i,j=1}^n \frac{1}{4} b_{ij} \rho_i^k\rho_j^k(w_i-w_j)^2  +  \sum_{j=1}^n \rho_j \log \rho_j \, dx
\\
	 &+ \int_{\mathbb{T}^d} \alpha \sum_{i=1}^n \rho_i^k w_i dx 
	+ \int_{\mathbb{T}^d} \sum_{i=1}^n \left( \beta_i (\rho_i - \rho_i^k) -  \nabla \beta_i  \cdot (\rho_i^k w_i)\right)dx,
\end{align*}
Computing the variational derivatives gives:
\begin{align*}
	\frac{\delta L}{\delta \rho_i} &= 0  \qquad \mbox{implies}   &&\log \rho_i + 1 + \beta_i = 0, \\
	\frac{\delta L}{\delta w_i} &= 0  \qquad \mbox{implies} &&\frac{1}{\Delta t}\sum_{j=1}^n b_{i j} \rho_i^k\rho_j^k(w_i-w_j) + \alpha \rho_i^k  - \rho_i^k \nabla \beta_i =0, \\ 
	\frac{\delta L}{\delta \alpha} &=  0  \qquad \mbox{implies}  &&\sum_{i=1}^n \rho_i^k w_i =0, \\
	\frac{\delta L}{\delta \beta_i} &=  0  \qquad \mbox{implies} &&\rho_i - \rho_i^k + \nabla \cdot (\rho_i^k  w_i) =0 .
\end{align*}
Taking 
$v_i = w_i/\Delta t$, 
we get
\begin{align*}
	\frac{\rho_i^{k+1}-\rho_i^k}{\Delta t} + \nabla \cdot (\rho_i^k v_i^{k+1}) &= 0, \\
	- \sum_{j=1}^n b_{i j} \rho_i^k\rho_j^k(v_i^{k+1}-v_j^{k+1}) &= \rho_i^k \nabla \log\rho_i^{k+1} - \frac{  \rho_i^k }{\sum_{j=1}^n \rho_j^k}\sum_{i=1}^n \rho_i^k \nabla \log \rho_i^{k+1},\\
	\sum_{i=1}^n \rho_i^k v_i^{k+1} &= 0.
\end{align*}
The latter corresponds to an implicit-explicit discretization in time of the system \eqref{eq:f1}-\eqref{eq:consm}.

Next we will give details of the  optimization formulation for the fully discretized scheme \eqref{eq:sc1}-\eqref{eq:sc3}.


We  prove the following theorem. 
\begin{theorem}\label{theorem}
	Assume $b_{ij}>0$ and  $b_{ij}=b_{ji}$ for $i\neq j$ and $i,j=1,\ldots,n$.
	Given
	$\rho^k \in \C_{\rm per}$ with $\rho^k > 0$. There exists $\delta_0>0$ such 
	that $\rho^{k+1}>0$ is a solution of the numerical scheme \eqref{eq:sc1}-\eqref{eq:sc3} if and only if it is a minimizer of the optimization problem:
	\begin{align}\label{eq:op}
	&\rho^{k+1} = \argmin_{(\rho,w) \in K_\delta} \left\{J= \frac{1}{4 \Delta t}  \left[\sum_{i,j=1}^n b_{ij}\hat\rho_i^k\hat\rho_j^k(w_i-w_j)^2 \right] + F_h(\rho)\right\},
\end{align}	
where
$ F_h(\rho)= \left\langle\sum_{i=1}^n \rho_i \log \rho_i \right\rangle,$
and
\begin{align*}
	K_\delta=\bigg\{ &(\rho,w):~ \rho \in \C_{\rm per}^n,~w \in \E_{\rm per}^n;~
	\rho_{i,\ell} \ge \delta,
	~~ 
	{\rho_{i,\ell}-\rho_{i,\ell}^k} + d_h(\hat\rho_i^k w_i)_{\ell}=0,\\
	&\sum_{i=1}^n \hat{\rho}_{i,\ell+\frac{1}{2}}^k w_{i,\ell+\frac12}=0 \text{ and } \sum_{i=1}^n \rho_{i,\ell} =1,~\forall i =1,\ldots,n,~ \forall \ell =1,\ldots,N\bigg\},
\end{align*}
for any $0<\delta\le\delta_0$.
\end{theorem}

We first prove a lemma that will be used later in the proof.

	\begin{lemma}\label{lm5}
	Suppose $\Phi$ is a $(n-1)\times (n-1)$ symmetric  positive definite matrix, with $\Phi_{ij}\in \E_{\rm per}$ for $i,j=1,\ldots,n-1$.
	Suppose $\phi \in \mathring{\C}^{n-1}_{\rm per}$
	 is bounded in $L^\infty$ satisfying $\|\phi\|_{L^\infty}
	  \le M$, where $\|\cdot\|_{L^\infty}$ is defined by
	\[\|\phi\|_{L^\infty}:= \max_{\substack{i=1,\ldots,n-1\\\ell=1,\ldots,N}} |\phi_{i,\ell}|.\]
	Then the following estimate holds
	\begin{align*}
		\|\mathcal{L}_{\Phi}^{-1} \phi\|_{L^\infty} 
		\le \frac{CM}{\lambda_{\min}} h^{-\frac{1}{2}}(n-1)^{\frac{1}{2}},
	\end{align*}
	where $C>0$ is a constant independent of $h$,
	$\lambda_{min}$ the minimum  of all eigenvalues of $\Phi$:
	\[\lambda_{\min} = \min_{\ell=1,\ldots,N} \left\{\lambda_{\ell}: \lambda_{\ell} \text{ is  the eigenvalue of } (\Phi_{ij,\ell+\frac12})_{(n-1)\times(n-1)}\right\}.\]
\end{lemma}
\begin{proof}
	Since $\|\phi\|_{L^\infty}\le M$,
	\begin{align*}
	    \|\phi\|_{L^2}^2: =& h \sum_{\substack{i=1,\ldots,n-1\\\ell=1,\ldots,N}} | \phi_{i,\ell}|^2 = h \sum_{\substack{i=1,\ldots,n-1\\\ell=1,\ldots,N}} |M|^2 \le (n-1)hN|M|^2 = (n-1)L|M|^2.
	\end{align*}
	Set $ g = \phi \in \mathring{\C}_{\rm per}^{n-1}$, and $f=\mathcal{L}_{\Phi}^{-1} g$ in \eqref{eq:linversenorm}, we get
	\begin{align*}
		&\|\phi\|_{\mathcal{L}_{\Phi}^{-1}}^2 = [D_h {f}, \Phi D_h {f}].
	\end{align*}
	Since $\Phi$ is positive definite so its minimum eigenvalues $\lambda_{\min}>0$,
	we get
	\begin{align*}
		&\lambda_{\min} \|D_h {f}\|_{L^2}^2 \le [D_h {f}, \Phi D_h {f} ] = -\langle {f}, d_h(\Phi D_h {f}) \rangle = \langle {f}, \phi\rangle \le \|{f}\|_{L^2} \|\phi\|_{L^2}.
	\end{align*}
	The use of the discrete Poincar\'e inequality gives
		$\|{f}\|_{L^2} \le C_P\|D_h {f}\|_{L^2}.$
	Therefore, we get
	\begin{align*}
		\|D_h {f}\|_{L^2} \le \frac{C_P}{\lambda_{\min}} \|\phi\|_{L^2}.
	\end{align*}
	Using an inverse inequality leads to
	\begin{align*}
		\|{f}\|_{L^\infty} \le C_1 h^{-\frac{1}{2}} \|D_h {f}\|_{L^2} \le \frac{C_1C_P}{\lambda_{\min}} h^{-\frac{1}{2}} L^{\frac{1}{2}}M(n-1)^{\frac{1}{2}} \le \frac{CM}{\lambda_{\min}}h^{-\frac{1}{2}}(n-1)^\frac12.
	\end{align*}
	\end{proof}

\begin{proof}[Proof of Theorem \ref{theorem}]
	The proof is divided into three steps. In the first two steps, we prove that the optimization problem \eqref{eq:op} has a unique interior minimizer and, in the last step, we prove that this minimizer is equivalent to the solution of the numerical scheme \eqref{eq:sc1}-\eqref{eq:sc3}.

	\emph{Step 1. Existence of the optimization problem.} 
	First we show existence for the optimization problem \eqref{eq:op} for any $\delta>0$. Notice that the objective function $J$ in \eqref{eq:op} is convex in $w$ but it is not strictly convex. However, we can rewrite the optimization problem by using the first $n-1$ components of $w$ and get an equivalent convex optimization problem. We introduce
	\begin{align*}
		W=(W_1,\ldots,W_{n}), \quad W_i = \hat\rho_i^k w_i,~~i=1,\ldots,n,
	\end{align*}
	and so 
		$\sum_{i=1}^n W_i = 0.$
We adopt the notation \eqref{eq:tildenotation} and define $\tilde W=(W_1,\ldots,W_{n-1})$.
	We have the following lemma.
	\begin{lemma}\label{lem:st6}
		The following formula holds:
		\begin{align}\label{eq:wformula}
		I(\tilde W) := \frac12\sum_{i=1}^n b_{ij}\hat\rho_i^k\hat\rho_j^k(w_i-w_j)^2= \tilde W^T(\hat Q^k)^T \hat B^k \hat Q^k \tilde W = \tilde{W}^T (\hat{D}^k)^{-1}\tilde{W}. 
		\end{align}
		For $\hat{\rho}^k >0$, the  function $I : \mathbb{R}^{n-1} \to \mathbb{R^+}$ is strictly convex.
	\end{lemma}
	\begin{proof}
		By the assumption that $b_{ij}$ is symmetric, the following formula holds
		\begin{align*}
				\frac{1}{2} \sum_{i,j=1}^n b_{ij}\hat\rho_i^k\hat\rho_j^k(w_i-w_j)^2=\sum_{i=1}^n w_i \sum_{j=1}^n b_{ij}\hat{\rho}_i^k \hat{\rho}_j^k(w_i-w_j).
		\end{align*}
		Recalling \eqref{eq:tm11}, \eqref{eq:asd}, we also have
		\begin{align*}
			\sum_{j=1}^n b_{ij}\hat{\rho}_i^k \hat{\rho}_j^k(w_i-w_j) = \sum_{j,m=1}^{n-1} \hat{B}^k_{ij} \hat{Q}_{jm}^k \hat{\rho}_m w_m
		\end{align*}
		Therefore,
		\begin{align*}
			\frac{1}{2} \sum_{i,j=1}^n & b_{ij}\hat\rho_i^k\hat\rho_j^k(w_i-w_j)^2 \\
			=& \sum_{i=1}^n w_i \sum_{j,m=1}^{n-1} \hat{B}^k_{ij} \hat{Q}_{jm}^k \hat{\rho}_m w_m \\
			=& \sum_{i=1}^{n-1} w_i \sum_{j,m=1}^{n-1} \hat{B}^k_{ij} \hat{Q}_{jm}^k \hat{\rho}_m w_m - \sum_{s=1}^{n-1} \frac{\hat{\rho}^k_s w_s}{\hat{\rho}^k_n} \sum_{j,m=1}^{n-1} \left(-\sum_{i=1}^{n-1}\hat{B}^k_{ij}\hat{Q}_{jm}^k \hat{\rho}^k_m w_m\right) \\
			=&\sum_{s,i,j,m=1}^{n-1} \hat{\rho}^k_s w_s \left(\frac{\delta_{is}}{\hat{\rho}^k_s} + \frac{1}{\hat{\rho}^k_n}\right)\hat{B}_{ij}^k\hat{Q}^k_{jm} \hat{\rho}^k_m w_m \\
			=& \sum_{s,i,j,m=1}^{n-1} \hat{\rho}^k_s w_s \hat{Q}^k_{is} \hat{B}_{ij}^k\hat{Q}^k_{jm} \hat{\rho}^k_m w_m
			= \tilde{W}^T (\hat{Q}^k)^T \hat{B}^k\hat{Q}^k \tilde{W}. 
		\end{align*}
			Notice that $\hat{B}^k$ is a symmetric strictly diagonally dominant matrix with  positive diagonal entries since $\rho^k>0$ and thus is positive definite. Because of this and since $\hat Q^k$ is non-singular, we have
	\begin{align*}
		(\hat Q^k)^T \hat B^k \hat Q^k \text{ is positive definite}.
	\end{align*}
	Therefore, \eqref{eq:wformula} is a convex function of $\tilde W$. 
	\end{proof}

	We also need a lemma on the convexity of the discretized energy function $F_h(\tilde \rho)$, defined by \eqref{lem:7}
	    that incorporates the constraint $\sum_{i =1}^n \rho_i = 1$.
	
	\begin{lemma}\label{lem:st7}
	    The energy function $F_h=F_h(\tilde{\rho})$ is a convex function of $\tilde{\rho}$. 
	\end{lemma}
	\begin{proof}
	    Considering the function
	    \begin{align*}
	        f =\sum_{i=1}^{n-1} \rho_i \log \rho_i + \rho_n \log \rho_n,\quad \rho_n = 1-\sum_{i=1}^{n-1} \rho_i,
	    \end{align*}
	    we have
	    \begin{align*}
	        &\frac{\partial f}{\partial \rho_i} = \log \rho_i +1 - (\log \rho_n +1) = \log \rho_i - \log \rho_n ,~~\frac{\partial^2 f}{\partial \rho_i \partial \rho_j} = \frac{1}{\rho_i}\delta_{ij} + \frac{1}{\rho_n}.
	    \end{align*}
	    Since for any $z\in \mathbb{R}^{n-1}$ and $z\not=0$,  
	    \begin{align*}
	        &\sum_{i,j=1}^{n-1} \frac{\partial^2 f}{\partial \rho_i \partial \rho_j}z_i z_j = \sum_{i,j=1}^{n-1} \left( \frac{1}{\rho_i}\delta_{ij} + \frac{1}{\rho_n}\right)z_iz_j = \sum_{i=1}^{n-1}\frac{1}{\rho_i} z_i^2 + \frac{1}{\rho_n}\left(\sum_{i=1}^{n-1} z_i\right)^2 > 0,
	    \end{align*}
	    the function $f$ is a convex function of $\tilde{\rho}$. Therefore, $F_h(\tilde{\rho})$ is convex in $\tilde{\rho}$.
	\end{proof}
	Using Lemmas \ref{lem:st6} and \ref{lem:st7},  we deduce that the optimization problem \eqref{eq:op} is equivalent to 
	\begin{align}\label{eq:jj}
		&\min_{(\tilde{\rho},\tilde{W}) \in \tilde K_\delta} \left\{ J = \frac{1}{2 \Delta t} \left[ \tilde W^T(\hat Q^k)^T \hat B^k \hat Q^k  \tilde W\right] + F_h(\tilde{\rho})\right\},
	\end{align}
	where
	\begin{align*}
		\tilde K_\delta = \{(\tilde{\rho},\tilde{W}):~&\tilde\rho \in \C_{\rm per}^{n-1},\tilde{W} \in \E_{\rm per}^{n-1};  \rho_{i,\ell} \ge \delta ,
		~~ \sum_{i=1}^{n-1} \rho_{i,\ell} \le 1- \delta \text{ and }\\&\rho_{i,\ell} - \rho_{i,\ell}^k + d_h( W_i)_{\ell}=0,\; \forall i =1,\ldots,n-1,\ell=1,\ldots,N\}.
	\end{align*}
Due to the above lemmas, the objective function $J$ is a convex function of $\tilde{W}$ and $\tilde{\rho}$ 
 (note that $(\hat Q^k)^T \hat B^k \hat Q^k$ is a fixed matrix determined from the previous step). The domain $\tilde K_\delta$  is affine in $\tilde{W}$  and it is convex and bounded in
 $\tilde{\rho}$. The optimization problem \eqref{eq:jj} has a unique minimizer according to standard optimization theory \cite{boyd2004convex}. Since the  problems
  \eqref{eq:op} and \eqref{eq:jj} are equivalent, there also exists a unique solution to the optimization problem \eqref{eq:op}.

	\emph{Step2. The minimizer does not touch the boundary.}
	Next, we show that there exists a constant $\delta_0>0$ such that the solution of the optimization problem \eqref{eq:op} could not touch the boundary of $K_\delta$ for $\delta \le \delta_0$.
	Recall that on the set $\tilde K_\delta$,
	$$
	\rho_i - \rho_i^k + d_h( W_i)=0.
	$$
	Hence, if we set
	\begin{align*}
		\tilde W = \hat D^k D_h \tilde f,   \quad   \tilde g = \tilde{\rho}-\tilde{\rho}^k \in \mathring{\C}_{\rm per}^{n-1} ,
	\end{align*}
	then according to the definition \eqref{eq:linversenorm},
	\begin{align}
		 \left[ \tilde W^T(\hat Q^k)^T \hat B^k \hat Q^k  \tilde W\right] 
		 = [( D_h \tilde f)^T  \hat D^k D_h \tilde f ] 
		 = \|\tilde{\rho}-\tilde{\rho}^k\|_{\mathcal{L}^{-1}_{\hat D^k}}^2.
		 \label{l2theory}
	\end{align}
	Therefore, the optimization problem \eqref{eq:jj} is equivalent to
	\begin{align}\label{eq:Jm}
		\min_{\tilde{\rho}\in \mathring{\tilde K}_\delta} \left\{J = \frac{1}{2 \Delta t}  \|\tilde{\rho}-\tilde{\rho}^k\|_{\mathcal{L}^{-1}_{\hat D^k}}^2 +  F_h(\tilde \rho)\right\},
	\end{align}
	over the set
	\begin{align*}
	    \mathring{\tilde K}_\delta = \bigg\{\tilde{\rho}:&~\tilde\rho - \tilde{\rho}^k \in \mathring{\mathcal{C}}_{\rm per}^{n-1};\rho_{i,\ell} \ge \delta,~\sum_{i=1}^{n-1} \rho_{i,\ell} \le 1-\delta,~\forall i =1,\ldots,n-1,\ell=1,\ldots,N\bigg\}.
	\end{align*}
 Recall the notation  $\tilde \rho = (\rho_1 , ... , \rho_{n-1})$ stands for the vector of the first $n-1$ densities which are computed 
at the grid points $l = 1, ..., N$. The density $\rho_n$ appears in the formulation \eqref{eq:Jm} only indirectly through the constraint \eqref{eq:consrho}.
Also,
 $\tilde\rho - \tilde{\rho}^k \in \mathring{\mathcal{C}}_{\rm per}^{n-1}$ means  $\sum_{\ell=1}^{N}(\rho_{i,\ell} - \rho_{i,\ell}^k) = 0$ for any $i = 1, ..., n-1$.

	Let $\tilde \rho^\star \in \mathring{\tilde K}_\delta$ be a minimizer of the optimization problem \eqref{eq:Jm}. We will show
	that  $\tilde \rho^\star$ does not lie on the boundary of $\mathring{\tilde K}_\delta$. If it lies on the boundary:
	\begin{enumerate}[(i)]
	    \item \label{en:1} either $\rho_{i,\ell}^\star =\delta$ for some $i=1,\ldots,n-1$ at some grid point $\ell$,
	    \item \label{en:2} or  $\sum_{i=1}^{n-1}  \rho^\star_{i,\ell} = 1-\delta$ at some grid point $\ell$.
	\end{enumerate}
 First consider the case \eqref{en:1}. Suppose that $\tilde\rho^\star$ touches the boundary at 
 the grid point $\ell_0$ 
 for the $i_0$-th component, that is 
	\begin{align}\label{eq:asdelta1}
	 	\rho^\star_{i_0,\ell_0}=\delta.
	 \end{align} 
	 We calculate the directional derivative of the objective function $J$ at $\tilde\rho^\star$ along the direction $\{\nu:\nu \in \mathbb{R}^{(n-1)\times N}\}$ with $\tilde\rho^\star + s\nu \in \mathring{\tilde K}_\delta$ as
	 \begin{align}\label{eq:Jdev}
	 	&\left. \frac{d}{ds} J(\tilde\rho^\star + s\nu) \right|_{s=0} \\
	 	&\quad = \left. \frac{d}{ds} \right|_{s=0} \left( \frac{1}{2 \Delta t}  \|\tilde{\rho}^\star + s\nu -\tilde{\rho}^k\|_{\mathcal{L}^{-1}_{\hat D^k}}^2 +  F_h(\tilde \rho^\star+s\nu)\right)\nonumber \\
	 	&\quad = \frac{1}{\Delta t} \left\langle \mathcal{L}^{-1}_{\hat{D}^k}(\tilde\rho^\star-\tilde\rho^k), \nu \right\rangle  + \sum_{i=1}^{n-1}\left\langle \log\rho^\star_i+1 - \log \left(1-\sum_{i=1}^{n-1}\rho^\star_i\right)-1, \nu_i\right\rangle \nonumber\\
	 	&\quad = \frac{1}{\Delta t} \left\langle \mathcal{L}^{-1}_{\hat{D}^k}(\tilde\rho^\star-\tilde\rho^k), \nu \right\rangle  + \sum_{i=1}^{n-1}\left\langle \left(\log \rho^\star_i- \log \left(1-\sum_{i=1}^{n-1}\rho^\star_i\right)\right), \nu_i \right\rangle. \nonumber
	 \end{align}
	 	 We divide into the following two cases:
	 \begin{enumerate}[(a)]
	     \item \label{en:a} $$\sum_{i=1}^{n-1} \rho^\star_{i,\ell_0} \ge \frac{1}{2},$$
	     \item \label{en:b} $$\sum_{i=1}^{n-1} \rho^\star_{i,\ell_0} < \frac{1}{2}. $$
	 \end{enumerate}
	
\smallskip
	  {\bf Case \eqref{en:1} and \eqref{en:a}}. Suppose $\{\rho^\star_{i,\ell_0}\}_{i=1}^{n-1}$ achieves its maximum at the $i_1$-th component
	  while $\{\rho^\star_{i_0,\ell}\}_{\ell=1}^N$ achieves its maximum at $\ell_1$. Define $\nu$ by
	 \begin{align*}
	 	\nu_{i,\ell}=\left\{\begin{array}{cl}
	 		1, &\text{for }i=i_0,\; \ell=\ell_0,\\
	 		-1,&\text{for }i=i_1,\; \ell=\ell_0,\\
	 		-1,&\text{for }i=i_0,\; \ell=\ell_1,\\
	 		1,&\text{for }i=i_1,\; \ell=\ell_1, \\
	 		0, &\text{otherwise}.
	 	\end{array}\right.
	 \end{align*}
	 Taking a variation in this direction, \eqref{eq:Jdev} becomes
	  \begin{align}\label{eq:Jderiv}
	 	&\left. \frac{1}{h}\frac{d}{ds} J(\tilde\rho^\star + s\nu) \right|_{s=0} \\
	 	&~ = \frac{1}{\Delta t} (\mathcal{L}^{-1}_{\hat{D}^k}(\tilde\rho^\star-\tilde\rho^k))_{i_0,\ell_0} -\frac{1}{\Delta t} (\mathcal{L}^{-1}_{\hat{D}^k}(\tilde\rho^\star-\tilde\rho^k))_{i_1,\ell_0} - \frac{1}{\Delta t}(\mathcal{L}^{-1}_{\hat{D}^k}(\tilde\rho^\star-\tilde\rho^k))_{i_0,\ell_1} \nonumber\\
	 	&\qquad+ \frac{1}{\Delta t} ((\mathcal{L}^{-1}_{\hat{D}^k}(\tilde\rho^\star-\tilde\rho^k))_{i_1,\ell_1}  + \log \rho^\star_{i_0,\ell_0} - \log \rho^\star_{i_1,\ell_0} -\log \rho^\star_{i_0,\ell_1} + \log \rho^\star_{i_1,\ell_1}. \nonumber
	 \end{align}
	
	 Since $\{\rho^\star_{i,\ell_0}\}_{i=1}^{n-1}$ achieves its maximum for the $i_1$-th component, in  the case \eqref{en:a} $\sum_{i=1}^{n-1}\rho^\star_{i,\ell_0} \ge \frac{1}{2}$ implies
	 \begin{align}\label{eq:cal1}
	 	\rho^\star_{i_1,\ell_0}  \ge \frac{1}{2(n-1)}.
	 \end{align}
	Since $\{\rho^\star_{i_0,\ell}\}_{\ell=1}^N$ achieves its maximum at the grid point $\ell_1$ and $\tilde\rho^\star-\tilde\rho^{k} \in \mathring{\C}_{\rm per}^{n-1}$,  
\begin{equation} \label{eq:cal2}
\rho^\star_{i_0,\ell_1} \ge \frac{1}{N}  \sum_{\ell=1}^N {\rho^\star_{i,\ell}} = \frac{1}{N} \sum_{\ell=1}^N {\rho^k_{i,\ell}}  \ge \frac{m}{hN}
\end{equation}
	 where $m$ is set to be
	 	$m := \min_{i\in\{1,\ldots,n-1\}}\left\{h\sum_{\ell=1}^N {\rho^k_{i,\ell}}\right\} \, .$
	 Moreover, for $\tilde\rho^\star \in \mathring{\tilde{{K}}}_\delta$ the constraint $\sum_{i=1}^{n-1}\rho^\star_{i,\ell_1}\le 1-\delta$ implies
	 \begin{align}\label{eq:cal3}
	 	\rho^\star_{i_1,\ell_1} < 1.
	 \end{align}
	 
	 Next, we show that for $\delta$ satisfying 
	 \begin{align}\label{restr1}
	  	\delta \le \min\left\{ \frac{m}{2hN},\frac{1}{4(n-1)}\right\},
	  \end{align} 
	  if $s >0$ is selected sufficiently small and $\nu$ as above we have  $\tilde\rho^\star + s\nu\in \mathring{\tilde K}_\delta$.
	  Indeed, 
	 \begin{align*}
	 	&\rho^\star_{i_0,\ell_0}+ s = \delta +s \ge \delta,~~\rho^\star_{i_1,\ell_1} + s \ge \delta + s,\\
	 	&\rho^\star_{i_0,\ell_1}-s \ge \frac{m}{hN} -s \ge \delta,~~
	 	\rho^\star_{i_1,\ell_0}-s \ge \frac{1}{2(n-1)}-s \ge \delta ,\\
	 	&\sum_{i=1}^{n-1} ( \rho^\star_{i,\ell_0} + s\nu_{i,\ell_0} ) = \sum_{i=1}^{n-1} \rho^\star_{i,\ell_0} \le 1-\delta,~~\sum_{i=1}^{n-1}( \rho^\star_{i,\ell_1} + s\nu_{i,\ell_1}) = \sum_{i=1}^{n-1} \rho^\star_{i,\ell_1}\le 1-\delta,
	 \end{align*}
	 imply that if $\delta$ satisfying \eqref{restr1} and for $s>0$ small we have $\tilde\rho^\star + s\nu \in \mathring{\tilde K}_\delta$.

	 Since $\tilde\rho^\star-\tilde{\rho}^k \in \mathring{\C}_{\rm per}^{n-1}$ and $\|\tilde{\rho}^\star\|_{L^\infty}, \|\tilde{\rho}^k\|_{L^\infty} \le 1$,
	 we can apply Lemma \ref{lm5} to \eqref{eq:Jderiv} with $\phi=\tilde\rho^\star-\tilde\rho^k$ and $\Phi=\hat{D}^k$ and use \eqref{eq:asdelta1}
	 and \eqref{eq:cal1}-\eqref{eq:cal3}  to get
	 \begin{align*}
	 	\left. \frac{1}{h}\frac{d}{ds} J(\tilde\rho^\star + s\nu) \right|_{s=0} &\le \frac{8C}{\lambda_{\min}^k\Delta t} h^{-\frac{1}{2}}(n-1)^{\frac12} + \log \delta - \log\frac{1}{2(n-1)} - \log\frac{m}{hN} +\log 1.
	 \end{align*}
	 Here $\lambda^k_{\min}$ is the minimum eigenvalue of $\hat{D}^k$.
	Taking
	 \begin{align}\label{eq:delta}
	 	\delta_0 \le \min\left\{\frac{m}{4(n-1)hN} e^{- \frac{8C}{\lambda_{\min}^k\Delta t} h^{-\frac{1}{2}}(n-1)^{\frac12}},\frac{m}{2hN},\frac{1}{4(n-1)}\right\},
	 \end{align}
	 we have for $\delta \le \delta_0$, $\tilde\rho^\star + s\nu \in \mathring{\tilde K}_\delta$ and
	 \begin{align}\label{eq:jle0}
	 	\left. \frac{1}{h}\frac{d}{ds} J(\tilde\rho^\star + s\nu) \right|_{s=0} \le -\log 2  <0.
	 \end{align}
	 This contradicts the assumption that $\tilde\rho^\star$ is a minimizer, and so the situation \eqref{en:a} cannot occur.
	 
\smallskip
	{\bf Case \eqref{en:1} and \eqref{en:b}}.
	 Again $\rho_{i_0, \ell_0} = \delta$ and suppose now that  $\{\rho^\star_{i_0,\ell}\}_{\ell=1}^N$ achieves its maximum at the $\ell_1$-th grid point.
	  We take	 
	 \begin{align*}
	 	\nu_{i,\ell}=\left\{\begin{array}{cl}
	 		1, &\text{for }i=i_0,\; \ell=\ell_0,\\
	 		-1, &\text{for }i=i_0,\; \ell=\ell_1, \\
	 		0, &\text{otherwise},
	 	\end{array}\right.
	 \end{align*}
	 and note that \eqref{eq:cal2} still holds in the present setting. Using \eqref{eq:asdelta1}, \eqref{en:b}, \eqref{eq:cal2}, and the inequality
	  $1-\sum_{i=1}^{n-1} \rho^\star_{i,\ell_1} \le 1 - (n-1) \delta \le 1,	$
	 we obtain
	 \begin{align*}
	 	&\left. \frac{1}{h}\frac{d}{ds} J(\tilde\rho^\star + s\nu) \right|_{s=0} \\
	 	&\quad = \frac{1}{\Delta t} (\mathcal{L}^{-1}_{\hat{D}^k}(\tilde\rho^\star-\tilde\rho^k))_{i_0,\ell_0}  +  \log \rho^\star_{i_0,\ell_0}- \log \left(1-\sum_{i=1}^{n-1}\rho^\star_{i,\ell_0}\right)\\
	 	&\qquad - \frac{1}{\Delta t} (\mathcal{L}^{-1}_{\hat{D}^k}(\tilde\rho^\star-\tilde\rho^k))_{i_0,\ell_1}  -  \log \rho^\star_{i_0,\ell_1}+ \log \left(1-\sum_{i=1}^{n-1}\rho^\star_{i,\ell_1}\right)\nonumber\\
	 	&\quad\le \frac{4C}{\lambda_{\min}^k\Delta t} h^{-\frac{1}{2}}(n-1)^{\frac12} + \log \delta - \log \frac{1}{2} -\log \frac{m}{hN} + \log 1 \\
	 	&\quad \le  \frac{4C}{\lambda_{\min}^k\Delta t} h^{-\frac{1}{2}}(n-1)^{\frac12} + \log \delta - \log \frac{m}{2hN} .
	 \end{align*}
	 Taking 
	 \begin{align}
	 	\delta_0 \le \min\left\{\frac{m}{4hN}e^{-\frac{4C}{\lambda_{\min}^k\Delta t} h^{-\frac{1}{2}}(n-1)^{\frac12} },\frac{m}{2hN}\right\}
	\label{eq:delta2}
	 \end{align}
	 leads to $\tilde\rho^\star + s\nu \in \mathring{\tilde K}_\delta$ and
	 \begin{align*}
	 	\left. \frac{1}{h}\frac{d}{ds} J(\tilde\rho^\star + s\nu) \right|_{s=0} =-\log 2<0,
	 \end{align*}
	 which contradicts the hypothesis that $\tilde\rho^\star$ is a minimizer; so the situation \eqref{en:b} cannot occur.
	
 \smallskip
	 {\bf Case \eqref{en:2}}.
	 Assume there exists a grid index $\ell_0$ such that
	 \begin{align}\label{eq:asi2}
	     \sum_{i=1}^{n-1}  \rho^\star_{i,\ell_0} = 1-\delta,
	 \end{align}
	 and suppose the maximum value of $\{\rho^\star_{i,\ell_0}\}_{i=1}^{n-1}$ occurs at the index $i_0$. Then \eqref{eq:asi2}
	 implies that for  $\delta \le 1/2$ equation \eqref{eq:cal1} holds, that is
	 \begin{align}\label{eq:ss1}
	   \rho^\star_{i_0,\ell_0} \ge \frac{1-\delta}{n-1} \ge \frac{1}{2(n-1)}.
	 \end{align}

	 	 Setting $\rho^k_{\min} := \min_{\substack{i=1,\ldots,n,\\ \ell=1,\ldots,N} }\rho^k_{i,\ell} >0$, 
	 we have
	 	$\sum_{i=1}^{n-1} \rho_{i,\ell}^k = 1 - \rho^k_{n,\ell} \le 1 - \rho^k_{\min}.$
	 Since $\tilde\rho^\star - \tilde\rho^k \in \mathring{C}_{\rm per}^{n-1}$, we have
	 \begin{align*}
	 	\sum_{\ell=1}^N \sum_{i=1}^{n-1} \rho^\star_{i,\ell} = \sum_{\ell=1}^N \sum_{i=1}^{n-1} \rho^k_{i,\ell} \le N(1-\rho^k_{\min}).
	 \end{align*}
	 Suppose $\left\{\sum_{i=1}^{n-1} \rho^\star_{i,\ell}\right\}_{\ell=1}^N$ achieves its minimum at the grid point $\ell_1$. Then
	 using \eqref{eq:asi2} it follows for $\delta \le \frac{1}{2}\rho_{\min}^k$,
	 \begin{align}
	 \sum_{i=1}^{n-1} \rho^\star_{i,\ell_1} &\le \tfrac{1}{N-1} \sum_{\substack{\ell = 1, ,,, , N  \\ \ell \ne \ell_0}} \sum_{i=1}^{n-1} \rho^\star_{i,\ell} 
	  \nonumber
	 \\
	 &= \tfrac{1}{N-1} \left ( \sum_{\ell = 1}^N  \sum_{i=1}^{n-1} \rho^\star_{i,\ell}  - \sum_{i=1}^{n-1} \rho^\star_{i,\ell_0} \right )
	 \nonumber
	 \\
	&\le \tfrac{1}{N-1}\left(N(1-\rho^k_{\min}) - (1-\delta)\right)
	 \nonumber
	\\
	&\le 1- \frac{ N \rho^k_{\min}-\delta}{N-1} 
	 \nonumber
	\\
	&\le 1-\frac{2N-1}{2(N-1)}\rho_{\min}^k \, .
	\label{eq:as3}
	 \end{align}
	Taking now
         \begin{align*}
	 	\nu_{i,\ell}=\left\{\begin{array}{cl}
	 		-1, &\text{for }i=i_0,\; \ell=\ell_0,\\
	 		1, &\text{for }i=i_0,\; \ell=\ell_1,\\
	 		0, &\text{otherwise},
	 	\end{array}\right.
	 \end{align*}
	 into \eqref{eq:Jdev}
	 and using  \eqref{eq:asi2}, \eqref{eq:ss1}, \eqref{eq:as3}, Lemma \ref{lm5}, and the inequality
	 $\rho^\star_{i_0,\ell_1}\le 1-\delta\le 1$ we obtain
	 \begin{align*}
	 	&\left. \frac{1}{h}\frac{d}{ds} J(\tilde\rho^\star + s\nu) \right|_{s=0} \nonumber\\
	 	&\quad = -\frac{1}{\Delta t} (\mathcal{L}^{-1}_{\hat{D}^k}(\tilde\rho^\star-\tilde\rho^k))_{i_0,\ell_0}  -  \log \rho^\star_{i_0,\ell_0}+ \log \left(1-\sum_{i=1}^{n-1}\rho^\star_{i,\ell_0}\right)\nonumber\\
	 	&\qquad+\frac{1}{\Delta t} (\mathcal{L}^{-1}_{\hat{D}^k}(\tilde\rho^\star-\tilde\rho^k))_{i_0,\ell_1}  +  \log \rho^\star_{i_0,\ell_1} - \log \left(1-\sum_{i=1}^{n-1}\rho^\star_{i,\ell_1}\right) \nonumber\\
	 	&\quad\le \frac{4C}{\lambda_{\min}^k\Delta t} h^{-\frac{1}{2}}(n-1)^{\frac12} - \log \frac{1}{2(n-1)} + \log \delta 
	 	+ \log 1 - \log \frac{2N-1}{2(N-1)} \rho_{\min}^k.
	 \end{align*}
	 	 
	 Taking 
	 \begin{align}\label{eq:delta0}
	 	\delta_0 \le \min\left\{\frac{(2N-1)\rho_{\min}^k}{8(N-1)(n-1)} e^{-\frac{4C}{\lambda_{\min}^k\Delta t} h^{-\frac{1}{2}}(n-1)^{\frac12}},
		 \frac{1}{2}\rho_{\min}^k,  \frac{1}{4(n-1)}  \right\},
	 \end{align}
	 we see that for $\delta < \delta_0$ the above inequality becomes negative. In addition, 
	 \begin{align*}
	 	&\rho^\star_{i_0,\ell_0} - s \ge \frac{1}{2(n-1)} -s \ge \delta, ~~\rho^\star_{i_0,\ell_1} + s \ge \delta +s \ge \delta, \\
	 	&\sum_{i=1}^n \rho^\star_{i,\ell_0} -s =1-\delta -s \le 1-\delta, ~~\sum_{i=1}^{n-1}\rho^\star_{i,\ell_1}+s \le 1-\frac{2N-1}{N-1} \delta +s \le 1-\delta,
	 \end{align*}
	  imply that for $\delta < \delta_0$ the variation $\tilde\rho^\star + sv \in \mathring{\tilde K}_\delta$ for sufficiently small $s>0$.
	  This contradicts the assumption that $\tilde\rho^\star$ is a minimizer and thus case \eqref{en:2} cannot occur.

	 In summary,  setting $\delta_0$ to be the minimum among \eqref{eq:delta}, \eqref{eq:delta2} and \eqref{eq:delta0} we conclude that
	 \eqref{en:1} and  \eqref{en:2} cannot occur. Consequently, for $\delta \le \delta_0$, the minimizer to the optimization problem \eqref{eq:Jm}, or equivalently \eqref{eq:op}, does not occur at the boundary.

	\emph{Step 3. The equivalence with the numerical scheme.}
	Any interior minimizer $\tilde\rho^{*}$ of \eqref{eq:Jm} must satisfies 
	\begin{align}\label{eq:pJ}
		\left\langle \frac{\partial J}{\partial \tilde\rho}(\tilde\rho^\star), \nu\right\rangle =0, 
	\end{align}
	for any $\nu \in \mathring{\C}_{\rm per}^{n-1}$ which is its tangent space, i.e., \eqref{eq:Jdev} equals zero. Due to the arbitary choice of $\nu$, we get
	\begin{align*}
		\frac{1}{\Delta t} \mathcal{L}^{-1}_{\hat D^k}(\tilde\rho^\star-\tilde\rho^k)_i + \log \rho^\star_i - \log \left(1-\sum_{j=1}^n \rho^\star_j\right)  = C_i,
	\end{align*}
	with $C_i, i=1,\ldots n-1$ being  constants,
	from which it follows that for $i=1,\ldots,n-1$,
	\begin{align*}
		\frac{\rho_i^\star-\rho_i^k}{\Delta t} =& -\mathcal{L}_{\hat D^k} \left(\log \tilde\rho^\star - \log \left(1-\sum_{j=1}^n \tilde\rho^\star_j\right)\right)_i 
		= \sum_{j=1}^{n-1} d_h(\hat D^k_{ij}D_h(\log \rho^\star_j - \log \rho_n^\star)).
	\end{align*}
	By Lemma \ref{lm4}, $\tilde\rho^\star$ satisfies the numerical scheme \eqref{eq:sc1}-\eqref{eq:sc3}.

	Conversely, assume $\rho^{k+1}>0$ is a solution of the numerical scheme \eqref{eq:sc1}-\eqref{eq:sc3}, we can reverse the above calculation with $C_i=0$ to show that \eqref{eq:pJ} holds, which together with the fact that the convex optimization problem \eqref{eq:Jm} has a unique interior minimizer, implies that $\rho^{k+1}$ is also the minimizer of \eqref{eq:Jm}, or equivalently of \eqref{eq:op}. 
	\end{proof}

\subsection{Properties of the scheme }
The positivity-preserving and energy stability properties of the scheme follow directly from Theorem \ref{theorem}.
\begin{theorem}\label{thm:prop}
	Assume $\rho^0$ defined in (\ref{ini}) is positive, the solution of the numerical scheme \eqref{eq:sc1}-\eqref{eq:sc2} then satisfies
	\begin{enumerate}
		\item (Positivity-preserving) $\rho^k >0$ for any $k\ge 1$,
		\item (Unconditionally energy stability) the inequality
		\begin{align}\label{eq:Energystability}
			F_h(\rho^{k}) + \|\tilde\rho^k - \tilde\rho^{k-1}\|^2_{\LL^{-1}_{\hat D^k}}  \le F_h(\rho^{k-1})
		\end{align}
		holds for any $k\ge1$.
	\end{enumerate}
\end{theorem}
\begin{proof}
	1. Starting from $\rho_0$, we apply Theorem \ref{theorem} recursively to obtain
	\begin{align*}
		\rho^{k} \in K_{\delta_k} 
	\end{align*}
	for some constant $\delta_k$ that is chosen for each step by the minimum among \eqref{eq:delta}, \eqref{eq:delta2} and \eqref{eq:delta0}. 
	This yields for every $k$,
	\begin{align*}
		\rho^k \in \bigcap_{k=1}^\infty K_{\delta_k} \subset K_{0}\backslash \{0\},
	\end{align*}
	so that $\rho^k>0$.

	2. Since the solution of the numerical scheme \eqref{eq:sc1}-\eqref{eq:sc3} is the minimizer of \eqref{eq:Jm}, we have
	\begin{align*}
		J(\rho^{k+1}) \le J(\rho^k),
	\end{align*}
	which is \eqref{eq:Energystability}.
\end{proof}

\section{Multidimensional case}\label{sec:mu}
The scheme can be generalized to the multidimensional case and similar proprties can be established.
Before we present the multi-dimensional scheme, we introduce some notations following \cite{wise2009energy}. Consider two multidimensional grids define by \hfill\break
\begin{align*}
	&\mathcal{C}^{d}  :=\underbrace{\mathcal{C}\times \cdots\times \mathcal{C}}_d, \quad \mathcal{E}_{x_s}:= \underbrace{\C\times\cdots \times\E\times \cdots \times \C}_d, ~s=1,\ldots,d,
\end{align*}
and the functions on them 
\begin{align*}
	\mathcal{C}^{d}_{\rm per}:=\{f:\mathcal{C}^d\to \mathbb{R}\}, \quad \mathcal{E}^{d}_{x_s,\rm per}:=\{f:\mathcal{E}_{x_s}^d\to \mathbb{R}\}, \quad \E^d_{\rm per}:=\left\{f:\bigcup_{s=1}^d \E_{x_s}^d\to\mathbb{R}\right\},
\end{align*}
as well as the vector functions,
$(\mathcal{C}_{\rm per}^d)^n:=\{f=(f_1,\ldots,f_n):f_i\in \mathcal{C}^d_{\rm per},i=1,\ldots,n\}$,
$(\mathcal{E}_{\rm per}^d)^n:=\{f=(f_1,\ldots,f_n):f_i\in {\E}^d_{\rm per},i=1,\ldots,n\}.$
We also define the space
\begin{align*}
 	(\mathring{\C}_{\rm per}^d)^n:=\left\{f\in (\mathcal{C}_{\rm per}^d)^n:\sum_{\ell \in \{1,\ldots,N\}^d}f_{i,\ell}=0,i=1,\ldots,n\right\}.
 \end{align*} 
We use $f_{\ell_1,\ldots,\ell_d}$ to denote the value of a function $f$ at the grid point $(x_1={\ell_1}h,\ldots,x_d={\ell_d}h)$.
We introduce the finite difference operators $D_h: \C_{\rm per}^d \mapsto \E_{\rm per}^d$ and $d_h: \E_{\rm per}^d \mapsto \C_{\rm per}^d$ as
\begin{align*}
	D_h f_{\ell_1,\ldots,\ell_s+\frac12,\ldots,\ell^d} = \frac{f_{\ell^1,\ldots,\ell^s+1,\ldots,\ell^d}-f_{\ell^1,\ldots,\ell^s,\ldots,\ell^d}}{h},
\end{align*}
and 
\begin{align*}
	d_h f_{\ell_1,\ldots,\ell_d}:=\sum_{s=1}^d \frac{f_{\ell^1,\ldots,\ell^s+\frac{1}{2},\ldots,\ell^d}-f_{\ell^1,\ldots,\ell^s-\frac{1}{2},\ldots,\ell^d}}{h}.
\end{align*}
We also define for $f \in \C_{\rm per}^d$,
	$\hat{f}_{\ell^1,\ldots,\ell^s+\frac{1}{2},\ldots,\ell^d} = \frac{f_{\ell^1,\ldots,\ell^s+1,\ldots,\ell^d} + f_{\ell^1,\ldots,\ell^s,\ldots,\ell^d}}{2},~~s=1,\ldots,d,$
so that $\hat{f} \in \E_{\rm per}^d$.
We define the inner products 
\begin{align*}
	&\langle f,g \rangle :=h^d\sum_{i=1}^n\sum_{\ell\in\{1,\ldots,N\}^d} f_{i,\ell} g_{i,\ell},~~ \forall f,g \in (\C^d_{\rm per})^n, \\
	&[f,g] := h^d \sum_{i=1}^n \sum_{\ell_1,\ldots,\ell_n=1}^N f_{i,\ell_1,\ldots,\ell_s+\frac12,\ldots,\ell_d}g_{i,\ell_1,\ldots,\ell_s+\frac12,\ldots,\ell_d},~~ \forall f,g \in (\E_{\rm per}^d)^n.
\end{align*}
The following summation-by-parts formula holds for any $f\in (\C^d_{\rm per})^n$ and $\phi\in (\E_{\rm per}^d)^n$,
\begin{align*}
	\langle f,d_h \phi\rangle = -[D_h f,\phi].
\end{align*}
Next we define a norm on $(\mathring{\C}_{\rm per}^d)^{n-1}$. Suppose $\Phi$ is a $(n-1)\times(n-1)$ symmetric positive definite matrix, with $\Phi_{ij} \in \E^d_{\rm per}$. We introduce the following operator 
\begin{align*}
	\LL_{\Phi} f = -d_h(\Phi D_h f) = -\sum_{j=1}^n d_h(\Phi_{ij} D_h f_j),
\end{align*}
where the multiplication $\Phi_{ij} D_h f_j$ is taken elementwise on the grid points. For any $g\in (\mathring{\C}_{\rm per}^d)^{n-1}$, let $f$ be determined by $g=\mathcal{L}_\Phi f$, we define the following norm
\begin{align}\label{eq:ginv}
	\|g\|_{\mathcal{L}_\Phi^{-1}}^2: = [D_h f,\Phi D_h f].
\end{align}

With the above notations, the numerical scheme for the system \eqref{eq:f1}-\eqref{eq:f2} is
\begin{align}
	\frac{\rho_i^{k+1}-\rho_i^k}{\Delta t} + d_h(\hat{\rho}_i^k v_i^{k+1})=&0,\label{eq:scd1}\\
	D_h \log \rho_i^{k+1} - \frac{1}{\sum_{i=1}^n \hat{\rho}_i^k}\sum_{j=1}^n \hat{\rho}_j^k D_h \log \rho_i^{k+1}  = &-\sum_{j=1}^n b_{ij}  \hat{\rho}_j^k (v_i^{k+1}-v_j^{k+1}),\label{eq:scd2}\\
	\sum_{i=1}^n \hat{\rho}_i^k v_i^{k+1} =& 0,\label{eq:scd3}
\end{align}
subject to initial data 
\begin{align}\label{ini+}
\rho_{i, \ell}^0=\rho_{i0}(x_\ell), \quad i=1,\ldots, n, \quad \ell=\{1,\ldots, N\}^d.
\end{align}
All properties proved for the one dimensional case carry over the $d$-dimensional case. The following theorem holds.
 \begin{theorem} \label{thm:propd}
 	Suppose $\rho^0>0$. 
	The solution of the numerical scheme \eqref{eq:scd1}-\eqref{eq:scd3} satisfies
 	\begin{enumerate}
 		\item (Conservation of mass.) For $k\ge 1$,
 		\begin{align*}
 			\sum_{i=1}^n  \rho_{i,\ell}^k = \sum_{i=1}^n\rho_{i,\ell}^0,~~ \text{ for all }\ell\in\{1,\ldots,d\}^N,
 		\end{align*}
 		and 
 		\begin{align*}
 			\sum_{\ell\in\{1,\ldots,d\}^N}\rho_{i,\ell}^k = 	\sum_{\ell\in\{1,\ldots,d\}^N}\rho_{i,\ell}^0,~~\text{ for all }i=1,\ldots, n.
 		\end{align*}
 		\item (Positivity-preserving.) For $k \ge 1$,
 		\begin{align*}
 			\rho^k >0.
 		\end{align*}
 		\item (Unconditional energy stability.) For $k \ge 1$, the following inequality holds:
 		\begin{align*}
 			F_h(\rho^{k}) + \|\tilde\rho^k - \tilde\rho^{k-1}\|^2_{\LL^{-1}_{\hat D^k}}  \le F_h(\rho^{k-1}),
 		\end{align*}
 		where
 			$F_h(\rho):= \left\langle \sum_{i=1}^n \rho_i\log \rho_i\right\rangle .$
 	\end{enumerate}
 \end{theorem}
 The proof of this result is similar, and therefore deferred to Appendix A. 


\section{Numerical Examples}\label{sec:ex}

We numerically validate our theoretical findings using numerical examples in both one and two dimensions.
\subsection{One dimension} We consider the numerical example on the unit torus $\mathbb{T}=[0,1]$ and take the initial condition similar as in \cite{boudin2012mathematical} as
\begin{align*}
	&\rho_{10}(x) = \left\{\begin{array}{ccl}
		&0.8, &\text{ for } 0\le x<0.25 \\
		&1.6(0.75-x), &\text{ for } 0.25 \le x < 0.5,\\
		&1.6(x-0.25), &\text{ for } 0.5 \le x < 0.75,\\
		&0.8, &\text{ for } 0.75 \le x < 1,
	\end{array}\right. \\
	&\rho_{20}(x) = 1\times 10^{-4}, \\
	&\rho_{30}(x) = 1- \rho_{10}(x)-\rho_{20}(x).
\end{align*}
We take the parameter $(b_{ij})_{n\times n}$ in the model to be
\begin{align*}
	b_{12} =b_{13}= \frac{1}{0.833}, b_{23} = \frac{1}{0.168}.
\end{align*}
The mesh size is taken to be $h=0.01$ and time step $\Delta t=0.001$. 

Here we calculate for $500$ time steps and the solutions reach equilibrium. The solution over time and the solution at $x=0.5$ are plotted in Figure \ref{fig:aa}.
\begin{figure}
		\includegraphics[width=0.48\textwidth]{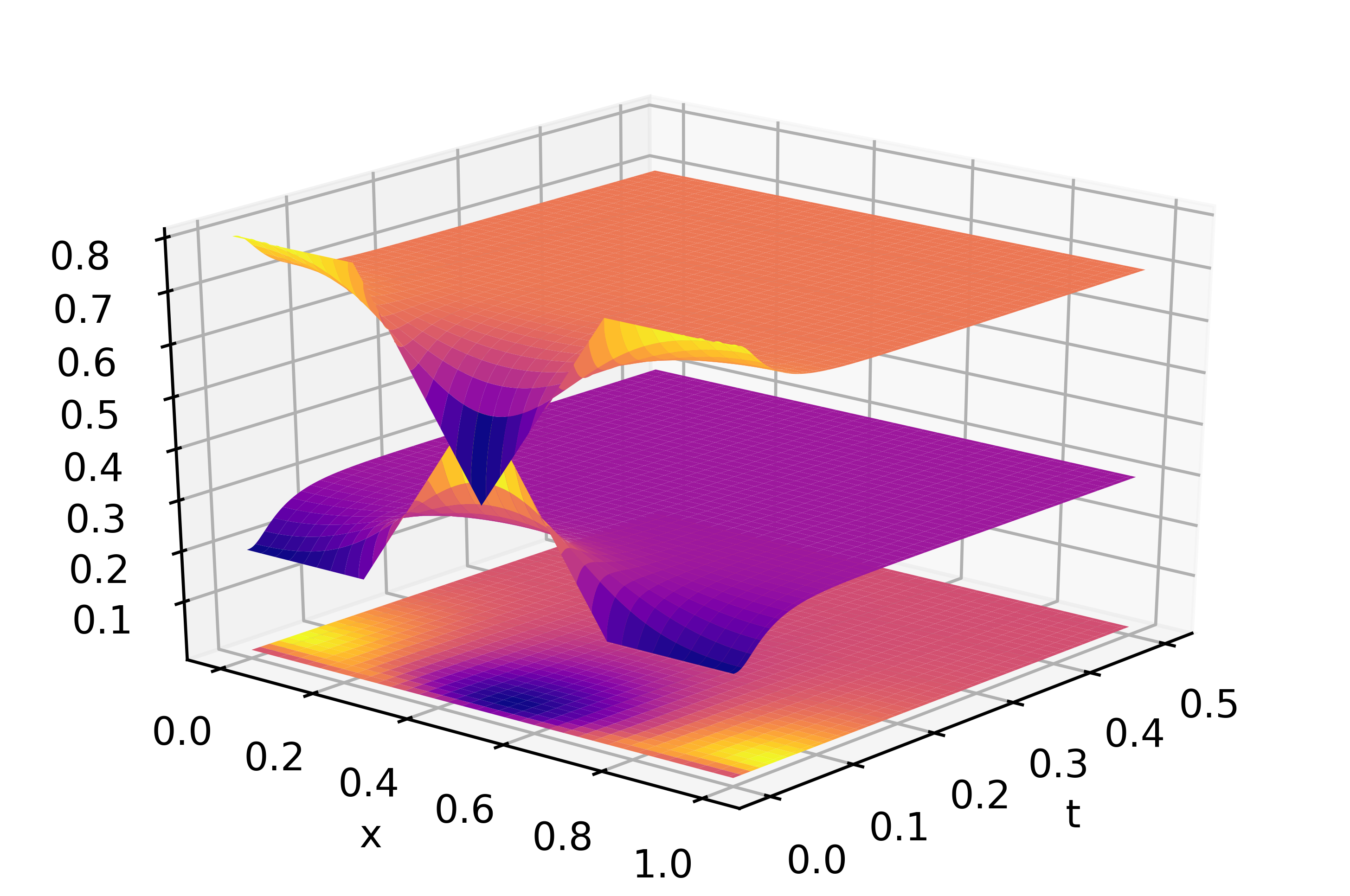}
	 \hfill
		\includegraphics[width=0.48\textwidth]{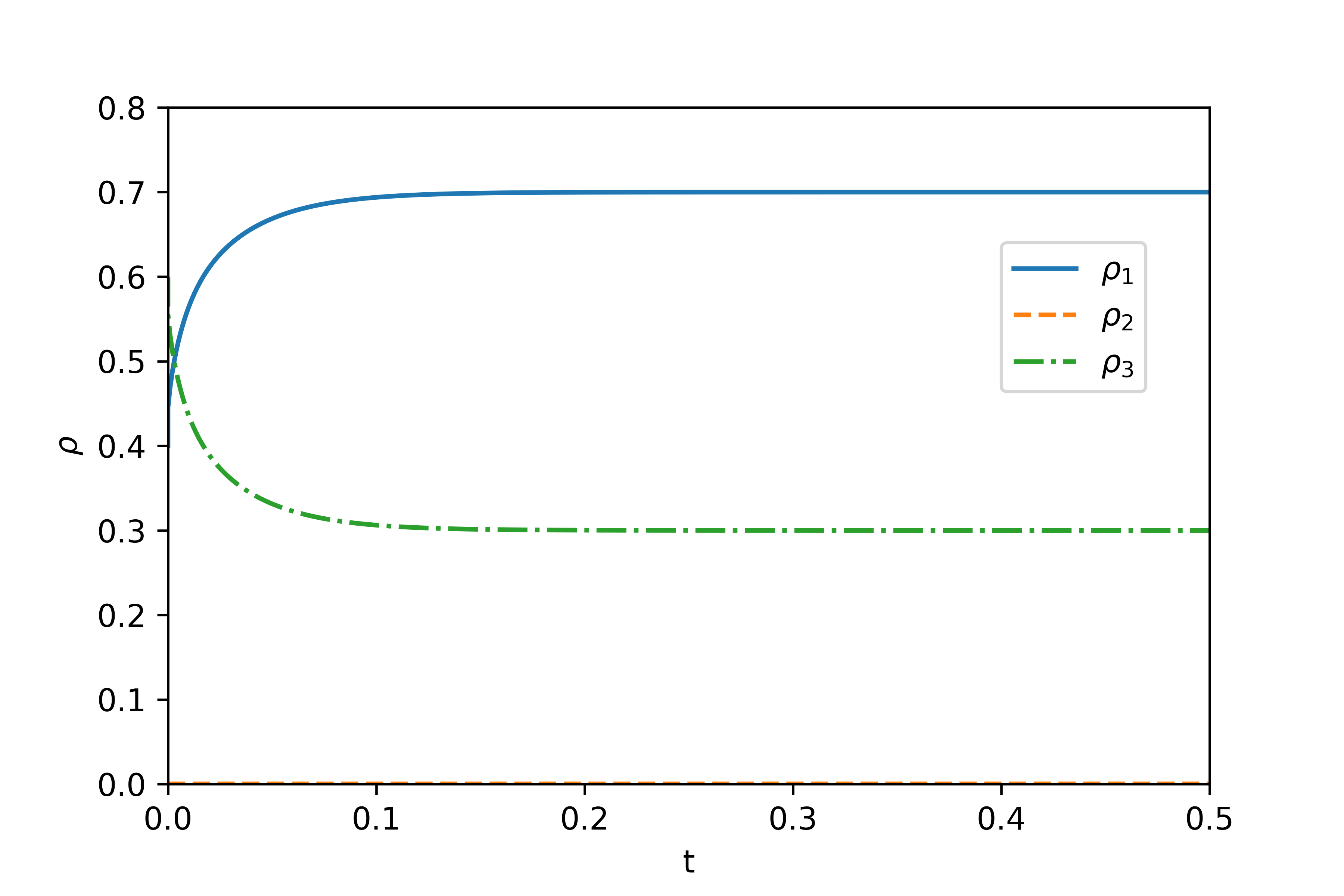}
	\caption{Result (left) and the solution (right) at $x=0.5$}
		\label{fig:aa}
\end{figure}
In our numerical test we observe that the variations of the mass defined in  Lemma \ref{lm1} and Lemma \ref{lm2} are of size $10^{-12} \sim 10^{-11}$,  which confirms the mass conservation results.  
The energy function $F_h(\rho)$ and the minimum value of $\rho$ are plotted in Figure \ref{fg:a1e}. 
\begin{figure}
	\includegraphics[width=0.48\textwidth]{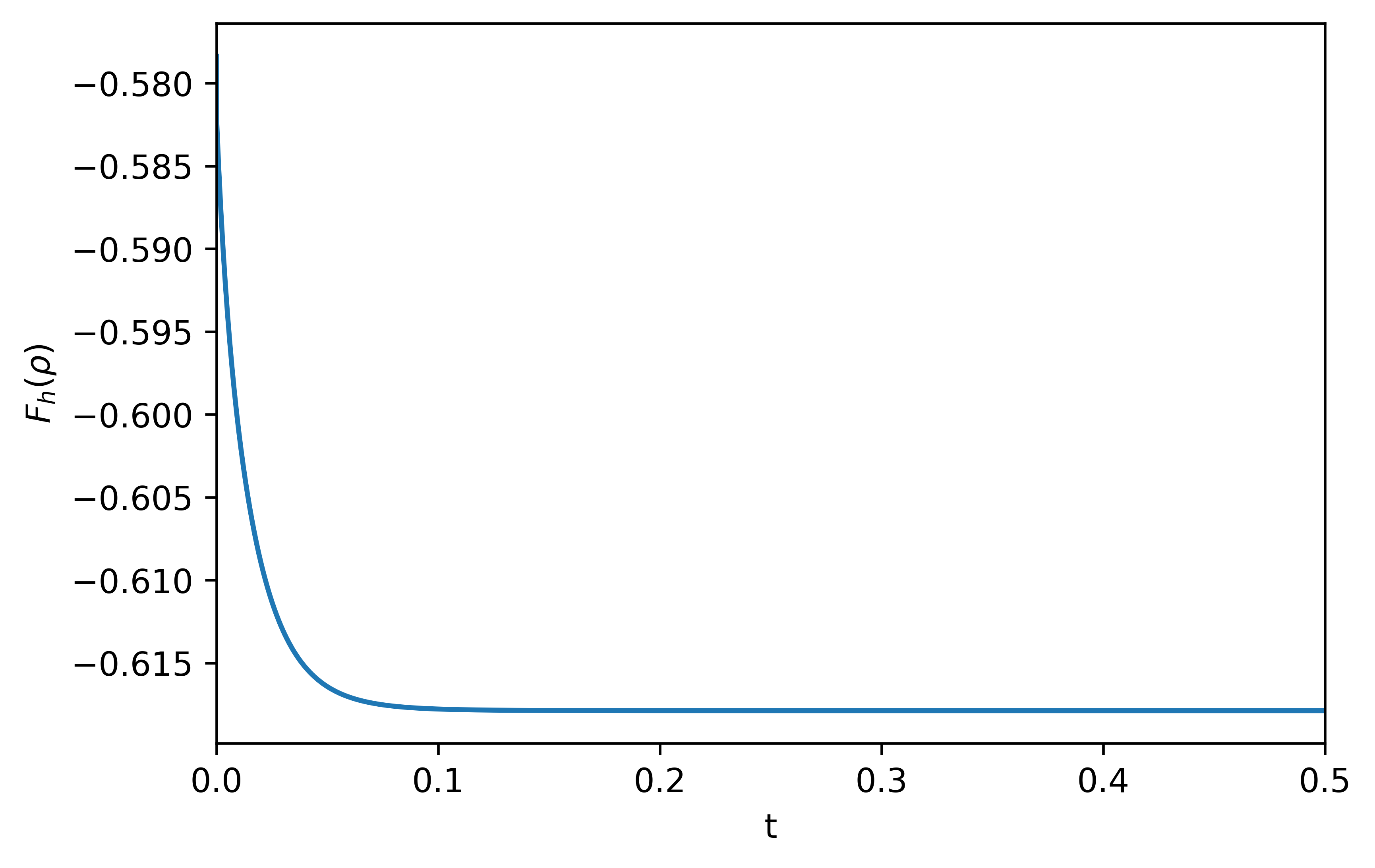}
	\hfill
	\includegraphics[width=0.48\textwidth]{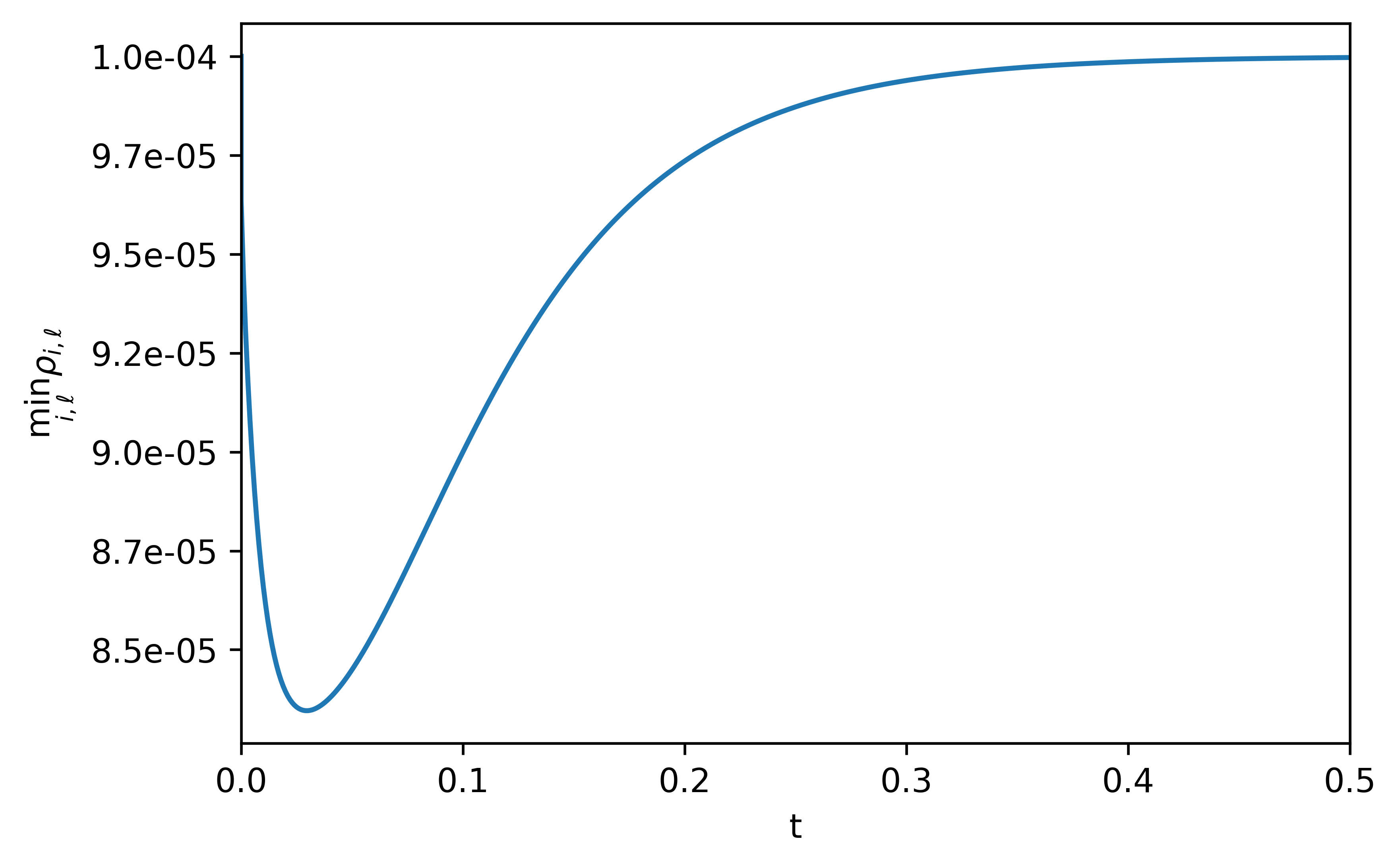}
	\caption{Energy (left) and Minimum value (right)}
	\label{fg:a1e}
\end{figure}
Theorem \ref{thm:prop} is verified.
We fix $\Delta t=0.01$ and calculate from $h=0.01$ to $h=0.2$ with $8$ values in equally distributed logrithmically. We plot the numerical error at $t=0.5$ with respect to the real solution $\rho_0(x)=(0.7,0.0001,0.299)$ in Figure \ref{fg:a1err}. 
\begin{figure}
	\includegraphics[width=0.48\textwidth]{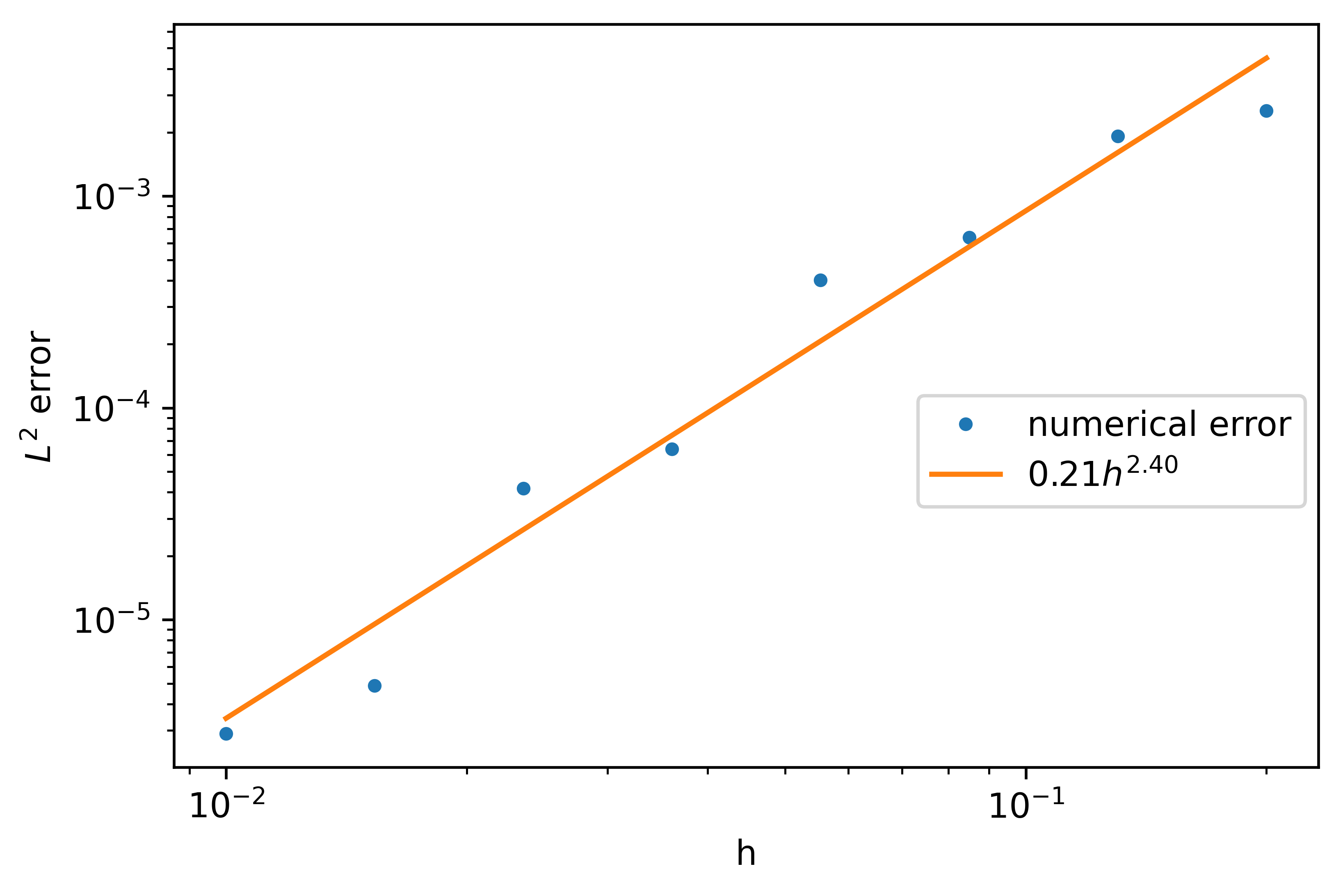} 
	\hfill
	\includegraphics[width=0.48\textwidth]{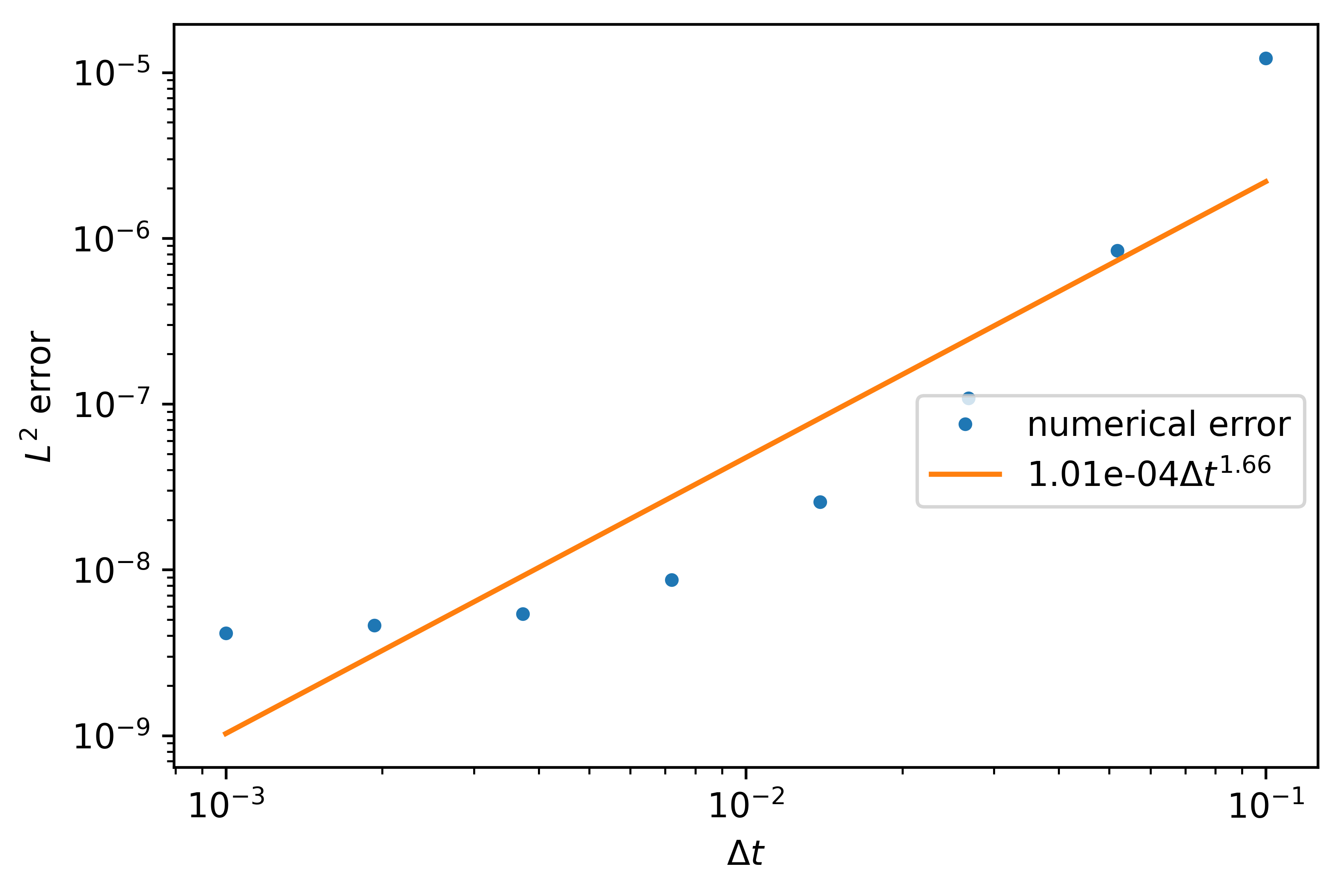} 
	\caption{Numerical errors}
	\label{fg:a1err}
\end{figure}
The fitted curve showed that the scheme is approximately of order $h^2$.
We also keep $h=0.01$ fixed and compute the numerical error with $\Delta t$ ranging from  $0.001$ to $0.1$. The result is plotted in Figure \ref{fg:a1err}. We see that the numerical error is approximately linear in $\Delta t$.


\subsection{Two dimensions}
We take 
\begin{align*}
	&\rho_{10}(x,y) = \left\{\begin{array}{ccl}
	&\frac{\sqrt{(x-\frac12)^2+(y-\frac12)^2}}{2}+ \frac{1}{10}, &\text{ for }\sqrt{(x-\frac12)^2+(y-\frac12)^2 }\le \frac{1}{8},\\
	&\frac{3}{5}, &\text{ otherwise},
	\end{array}\right. \\
	&\rho_{20}(x,y) = 1\times 10^{-4}, \\
	&\rho_{30}(x,y) = 1- \rho_{10}(x,y)-\rho_{20} (x,y).
\end{align*}
The mesh size is taken to be $h=0.05$ and time step $\Delta t=0.001$. We calculate for $500$ time steps. 
 The energy and minimum values are shown in Figure \ref{fg:a22}. We can see that the energy is decaying and the minimum values are all positive.
\begin{figure}
	\includegraphics[width=0.48\textwidth]{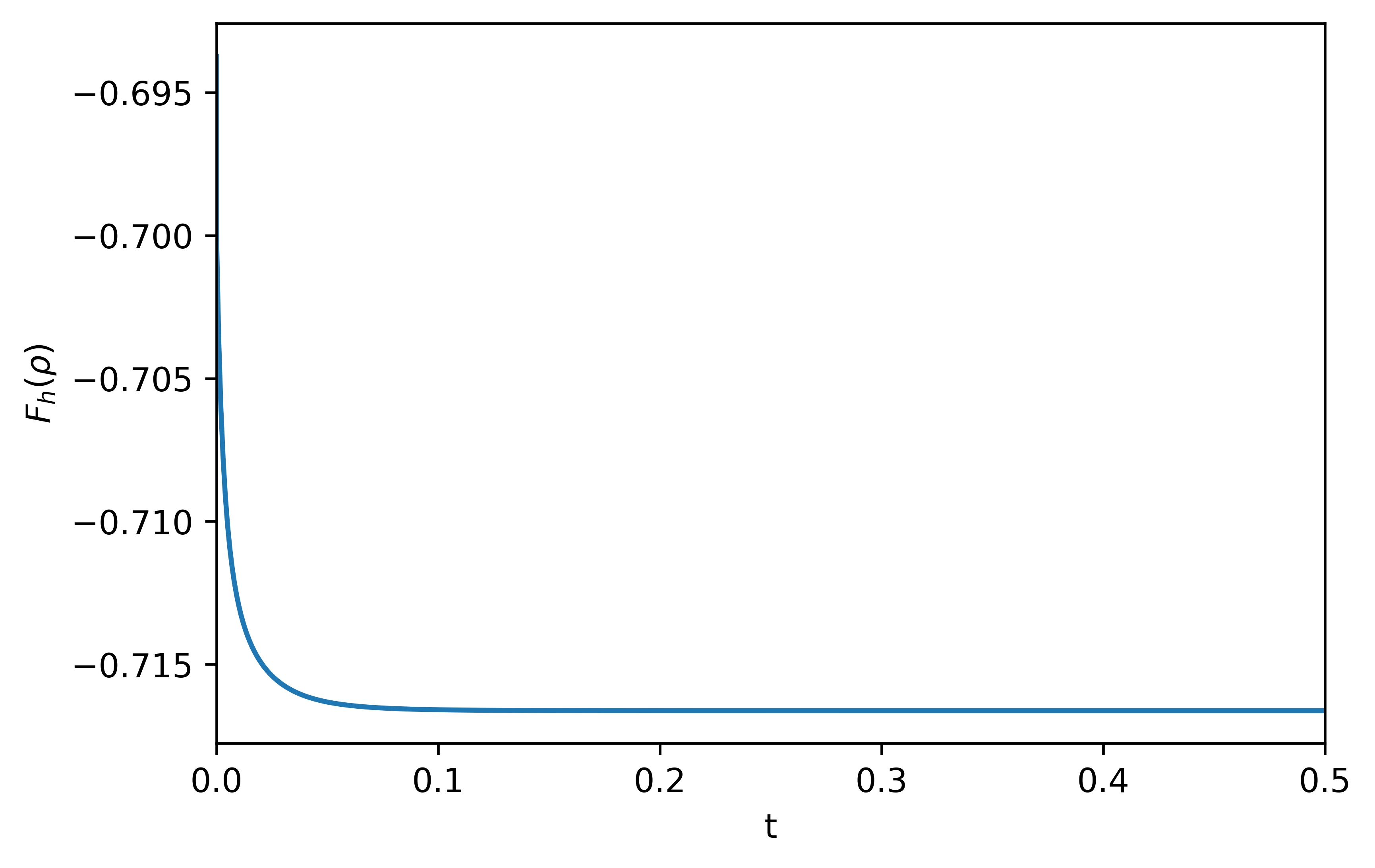}
	\hfill
	\includegraphics[width=0.48\textwidth]{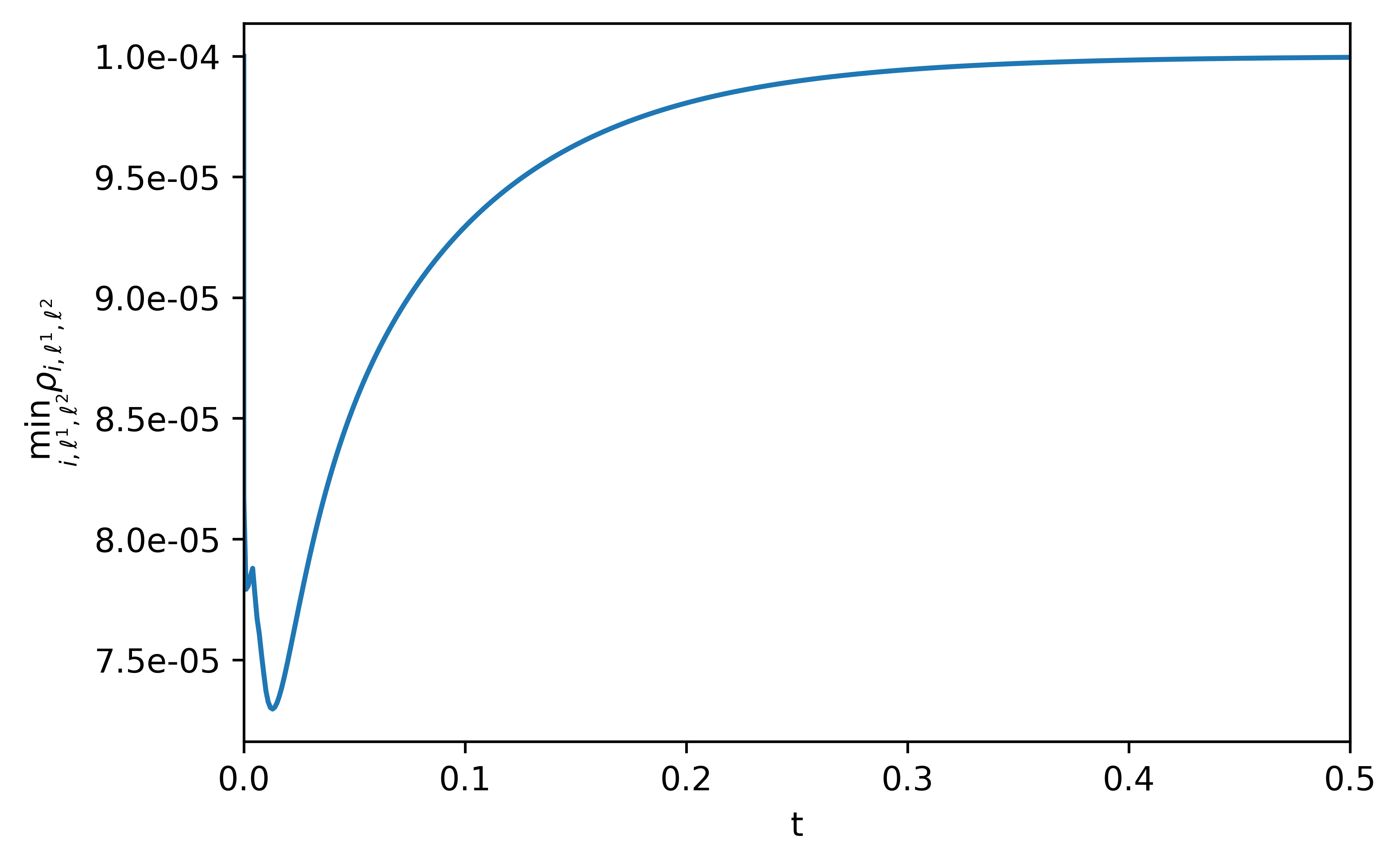}
	\caption{Energy (left) and Minimum value (right)}
	\label{fg:a22}
\end{figure}



\bigskip
        \appendix 
        
  \section{Proof Theorem \ref{thm:propd}}
  To prove Theorem \ref{thm:propd}, we need first to prove a multidimensional version of Theorem \ref{theorem}.
\begin{theorem}\label{eq:thmd}
	Assume $b_{ij}>0$  and $b_{ij}=b_{ji}$ for $i\neq j$ and $i,j=1,\ldots,n$. Assume $\rho^k \in (\C_{\rm per}^d)^n$ be positive. Then there exists a constant $\delta_0>0$, such that $\rho^{k+1}>0$ is a solution of the numerical scheme \eqref{eq:scd1}-\eqref{eq:scd3} if and only if it is a minimizer of the optimization problem:
	\begin{align}\label{eq:opd}
	&\rho^{k+1} = \argmin_{(\rho,w) \in K_\delta} \left\{J= \frac{1}{4 \Delta t}  \left[\sum_{i,j=1}^n b_{ij}\hat\rho_i^k\hat\rho_j^k(w_i-w_j)^2 \right] + F_h(\rho)\right\},
\end{align}	
where
\begin{align*}
	K_\delta=\bigg\{(\rho,w):~ &\rho \in (\C_{\rm per}^d)^n,~w \in (\E_{\rm per}^d)^n; ~ \rho_{i,\ell} \ge \delta,~~
	{\rho_{i,\ell}-\rho_{i,\ell}^k} + d_h(\hat\rho_i^k w_i)_{\ell}=0,\\
	&\sum_{i=1}^n \hat{\rho}_{i,\ell_1,\ldots,\ell_s+\frac12,\ldots,\ell_d}^k w_{i,\ell_1,\ldots,\ell_s+\frac12,\ldots,\ell_d}=0 \text{ and } 
	\sum_{i=1}^n \rho_{i,\ell} =1,\\
	&~\forall i =1,\ldots,n,~ \forall \ell =(\ell_1,\ldots,\ell_d) \in\{1,\ldots,N\}^d,s=1,\ldots,d\bigg\},
\end{align*}
for any $0<\delta\le\delta_0$.
\end{theorem}
The proof follows a similar strategy as the proof of Theorem \ref{theorem} for the one dimensional case. We establish a multidimensional version of Lemma \ref{lm5}.

	\begin{lemma}\label{lm6}
	Suppose $\Phi$ is a $(n-1)\times(n-1)$ symmetric positive definite matirx, with $\Phi_{ij} \in \E^d_{\rm per}$. Suppose $\phi \in (\mathring{\C}^d)^{n-1}_{\rm per}$ 
	satisfies $\|\phi\|_{L^\infty} \le M$,
	\[\|\phi\|_{L^\infty}:= \max_{\substack{i=1,\ldots,n-1\\\ell_s=1,\ldots,N\\s=1,\ldots,d}} |\phi_{i,\ell_1,\ldots,\ell_d}|.\]
	The following estimate holds
	\begin{align*}
		\|\mathcal{L}_{\Phi}^{-1} \phi\|_{L^\infty} \le \frac{CM}{\lambda_{\min}} h^{-\frac{1}{2}}(n-1)^{\frac{1}{2}},
	\end{align*}
	where $C>0$ depends only on the domain, $\lambda_{min}$  is the minimum  of the eigenvalues of $\Phi$ over all grid points:
	\[
	\lambda_{\min} = \min_{\substack{\ell_s=1,\ldots,N\\s=1,\ldots,d}} \left\{\lambda_{\ell_1,\ldots,\ell_s+\frac12,\ldots,\ell^d} \text{ the eigenvalue of } 
	 (\Phi_{ij,\ell_1,\ldots,\ell_s+\frac12,\ldots,\ell_d})_{(n-1)\times(n-1)}\right\}
	 \]
\end{lemma}
\begin{proof}
	\begin{align*}
	    \|\phi\|_{L^2}^2: =& h^d \sum_{\substack{i=1,\ldots,n-1\\\ell^s=1,\ldots,N\\s=1,\ldots,d}} |\phi_{i,\ell_1,\ldots,\ell_d}|^2 \\\le& h^d \sum_{\substack{i=1,\ldots,n-1\\\ell_s=1,\ldots,N\\s=1,\ldots,d}} |M|^2 \le (n-1)h^dN^d|M|^2 = (n-1)L^d|M|^2.
	\end{align*}
	Let $g = \phi$ and $f=\mathcal{L}_\Phi^{-1} g$ in \eqref{eq:ginv}, the norm 
	satisfies
	\begin{align*}
		&\lambda_{\min} \|D_h {f}\|_{L^2}^2 \le [D_h {f}, \Phi D_h f] \\
		&\quad= -\langle {f}, d_h(\Phi D_h {f}) \rangle = -\langle {f}, \phi\rangle \le \|{f}\|_{L^2} \|\phi\|_{L^2} \le C_P \|{f}\|_{L^2}\|\phi\|_{L^2},
	\end{align*}
	according to the discrete Poincar\'e inequality.
	Therefore, we get
	\begin{align*}
		\|D_h {f}\|_{L^2} \le \frac{C_P}{\lambda_{\min}} \|\phi\|_{L^2}.
	\end{align*}
	Using an inverse inequality in $(\mathring{\C}^d)_{\rm per}^{n-1}$ leads to 
	\begin{align*}
		\|{f}\|_{L^\infty} \le C_1 h^{-\frac{1}{2}} \|D_h {f}\|_{L^2} \le \frac{C_1C_P}{\lambda_{\min}} h^{-\frac{1}{2}} L^{\frac{d}{2}}M(n-1)^{\frac{1}{2}} \le \frac{CM}{\lambda_{\min}}h^{-\frac{1}{2}}(n-1)^\frac12.
	\end{align*}
	\end{proof}
Now we prove Theorem \ref{eq:thmd}.
\begin{proof}
	In a fashion similar to the proof of the one dimensional case, there exists a unique solution to the optimization problem \eqref{eq:opd} for any $\delta>0$. This follows from the same argument with notations replaced by the multidimensional version. To prove that the minimizer of \eqref{eq:opd} does not touch the boundary of $K_\delta$, we use the equivalent optimization problem 
	\begin{align}\label{eq:Jmd}
		\min_{\tilde{\rho}\in \mathring{\tilde K}_\delta} \left\{J =  \frac{1}{2 \Delta t} \|\tilde{\rho}-\tilde{\rho}^k\|_{\mathcal{L}^{-1}_{\hat D^k}}^2 +  F_h(\tilde \rho)\right\},
	\end{align}
	over the set
	\begin{align*}
	    \mathring{\tilde K}_\delta = \bigg\{\tilde{\rho}:&~\tilde\rho - \tilde{\rho}^k \in (\mathring{\mathcal{C}}_{\rm per}^d)^{n-1};~  \rho_{i,\ell} \ge \delta ,~
	     \sum_{i=1}^{n-1} \rho_{i,\ell} \le 1-\delta,\\
	     &~\forall i =1,\ldots,n-1,~~\ell\in\{1,\ldots,N\}^d\bigg\}.
	\end{align*}
	Assume the minimizer touches the boundary of $\mathring{\tilde K}_\delta$ at the grid point $\ell^0=({\ell^0_{1},\ldots,\ell^0_{d}})$ for the $i_0$-th component, i.e.
	\begin{align}\label{eq:multdelta}
		\rho^\star_{i_0,\ell^0_{1},\ldots,\ell^0_{d}}=\delta.
	\end{align}

Next we consider the following two cases:
	 \begin{enumerate}[(a)]
	     \item \label{en:aa} $$\sum_{i=1}^{n-1} \rho^\star_{i,\ell^0} \ge \frac{1}{2},$$
	     \item \label{en:bb} $$\sum_{i=1}^{n-1} \rho^\star_{i,\ell^0} < \frac{1}{2}. $$
	 \end{enumerate}

	First consider the case \eqref{en:aa}. We also suppose $\{\rho^\star_{i,\ell_1^0,\ldots,\ell_d^0}\}_{i=1}^{n-1} 
	$ achieves its maximum at the $i_1$-th component, and $\{\rho^\star_{i_0,\ell}\}_{\ell\in\{1,\ldots,N\}^d}$ achieves its maximum at $\ell=\ell^1 = (\ell^1_1,\ldots,\ell^1_d)$.
	We calculate the directional derivative of the objective function \eqref{eq:Jmd} along the direction 	 
	\begin{align*}
	 	\nu_{i,\ell_1,\ldots,\ell_d}=\left\{\begin{array}{cl}
	 		1, &\text{for }i=i_0,\; \ell_s=\ell_s^0,~\forall s=1,\ldots,d,\\
	 		-1,&\text{for }i=i_1,\; \ell_s=\ell_s^0,~\forall s=1,\ldots,d,\\
	 		-1,&\text{for }i=i_0,\; \ell_s=\ell_s^1,~~\forall s=1,\ldots,d, \\
	 		1, &\text{for }i=i_1,\; \ell_s=\ell_s^1,~~\forall s=1,\ldots,d,\\
	 		0, &\text{otherwise},
	 	\end{array}\right.
	 \end{align*}
	 and we get	
	  	\begin{align}
	 	&\left. \frac{1}{h^d}\frac{d}{ds} J(\tilde\rho^\star + s\nu) \right|_{s=0} \nonumber\\
	 	&\quad = \frac{1}{\Delta t} (\mathcal{L}^{-1}_{\hat{D}^k}(\tilde\rho^\star-\tilde\rho^k))_{i_0,\ell^0} - \frac{1}{\Delta t} (\mathcal{L}^{-1}_{\hat{D}^k}(\tilde\rho^\star-\tilde\rho^k))_{i_1,\ell^0} - \frac{1}{\Delta t} (\mathcal{L}^{-1}_{\hat{D}^k}(\tilde\rho^\star-\tilde\rho^k))_{i_0,\ell^1} \nonumber\\
	 	&\qquad+ \frac{1}{\Delta t} (\mathcal{L}^{-1}_{\hat{D}^k}(\tilde\rho^\star-\tilde\rho^k))_{i_1,\ell^1}  + \log \rho^\star_{i_0,\ell^0} - \log \rho^\star_{i_1,\ell^0} -\log \rho^\star_{i_0,\ell^1} + \log \rho^\star_{i_1,\ell^1}. \label{eq:Jderivd}
	 \end{align}
	 Since $\rho^\star_{i_1,\ell^0}$ is the maximum point and the assumption \eqref{en:aa} that
	 $\sum_{i=1}^{n-1}\rho^\star_{i,\ell^0} \ge \frac{1}{2}$,
	 \begin{align}\label{eq:c1}
	 	\rho^\star_{i_1,\ell^0} \ge \frac{1}{2(n-1)}.
	 \end{align}
	 Since $\rho^\star_{i_0,\ell^1}$ is the maximum point and 
	 \begin{align*}
	 	\sum_{\ell\in\{1,\ldots,d\}^N} \rho^\star_{i_0,\ell} = \sum_{\ell\in\{1,\ldots,d\}^N} \rho^k_{i_1,\ell},
	 \end{align*}
	 we have
	 \begin{align}\label{eq:c2}
	 	\rho^\star_{i_0,\ell^1} \ge \frac{m}{h^dN^d},
	 \end{align}
	 where $m$ is set to be
	 \begin{align*}
	 	m=\min_{\{i=1,\ldots,n-1\}}\left\{h^d\sum_{\ell\in\{1,\ldots,N\}^d} \rho^k_{i,\ell}\right\}.
	 \end{align*}
	 In order to guarantee $\tilde\rho^\star + s\nu \in \mathring{\tilde K}_\delta$, we assume
	 \begin{align*}
	  	\delta \le \frac{m}{2h^dN^d}
	  \end{align*} 
	  so that $\rho^\star_{i_0,\ell^1} -s \ge \frac{G}{h^dN^d} -s \ge \delta$ for small $s$. One can check for other components and get $\tilde\rho^\star+s\nu \in \mathring{\tilde K}_\delta$ for $\delta \le \frac{1}{4(n-1)}$.
	 We also have
	 \begin{align*}
	 	\rho^\star_{i_1,\ell^1} \le 1-\delta < 1.
	 \end{align*}
	 Taking the above inequality and \eqref{eq:multdelta}, \eqref{eq:c1}-\eqref{eq:c2} into \eqref{eq:Jderivd} and applying Lemma \ref{lm6} leads to
	 \begin{align*}
	 	&\left. \frac{1}{h^d}\frac{d}{ds} J(\tilde\rho^\star + s\nu) \right|_{s=0} 
		\\
		&\qquad \qquad \le \frac{8C}{\lambda_{\min}^k\Delta t} h^{-\frac{1}{2}}(n-1)^{\frac12} + \log \delta - \log\frac{1}{2(n-1)} - \log\frac{m}{h^dN^d} +\log 1,
	 \end{align*}
	 where $\lambda^k_{\min}$ is the minimum eigenvalue of $\hat{D}^k$. Taking
	 \begin{align}\label{eq:dchoice1}
	 	\delta_0 \le \min\left\{\frac{m}{4(n-1)h^dN^d} e^{- \frac{8C}{\lambda_{\min}^k\Delta t} h^{-\frac{1}{2}}(n-1)^{\frac12}},\frac{m}{2h^dN^d}, \frac{1}{4(n-1)}\right\}
	 \end{align}
	 leads to 
	 \[\left. \frac{1}{h^d}\frac{d}{ds} J(\tilde\rho^\star + s\nu) \right|_{s=0}\le -\log 2 <0,\]
	 which contradicts to the assumption that $\tilde\rho^\star$ is a minimizer. 

	 Next we consider the case \eqref{en:bb}. We also suppose $\{\rho^\star_{i_0,\ell}\}_{\ell\in\{1,\ldots,N\}^d}$ achieves its maximum at $\ell=\ell^1=(\ell^1_1,\ldots,\ell^1_d)$. We take
	 	\begin{align*}
	 	\nu_{i,\ell_1,\ldots,\ell_d}=\left\{\begin{array}{cl}
	 		1, &\text{for }i=i_0,\; \ell_s=\ell_s^0,~\forall s=1,\ldots,d,\\
	 		-1,&\text{for }i=i_1,\; \ell_s=\ell_s^0,~\forall s=1,\ldots,d,\\
	 		0, &\text{otherwise},
	 	\end{array}\right.
	 \end{align*}
	 and use \eqref{eq:multdelta}, \eqref{en:bb}, \eqref{eq:c2} to get	
	  	\begin{align*}
	 	&\left. \frac{1}{h^d}\frac{d}{ds} J(\tilde\rho^\star + s\nu) \right|_{s=0} \nonumber\\
	 	&\quad = \frac{1}{\Delta t} (\mathcal{L}^{-1}_{\hat{D}^k}(\tilde\rho^\star-\tilde\rho^k))_{i_0,\ell^0}  - \frac{1}{\Delta t} (\mathcal{L}^{-1}_{\hat{D}^k}(\tilde\rho^\star-\tilde\rho^k))_{i_0,\ell^1}  + \log \rho^\star_{i_0,\ell^0}  \\
	 	&\quad \quad - \log \left(1-\sum_{i=1}^{n-1} \rho^\star_{i_0,\ell_1} \right) -\log \rho^\star_{i_0,\ell_1} + \log \left(1-\sum_{i=1}^{n-1} \rho^\star_{i_0,\ell_1}\right) \\
	 	&\quad\le \frac{4C}{\lambda_{\min}^k\Delta t} h^{-\frac{1}{2}}(n-1)^{\frac12} + \log \delta - \log\frac{1}{2} - \log\frac{m}{h^dN^d} +\log 1, 
	 	\end{align*}
	 		 Taking 
	 \begin{align}\label{eq:delta0choice3}
	 	\delta_0 \le \min\left\{\frac{m}{4h^dN^d}e^{-\frac{4C}{\lambda_{\min}^kh\Delta t} h^{-\frac{1}{2}}(n-1)^{\frac12} },\frac{m}{2h^dN^d}\right\}
	 \end{align}
	 leads to 
	 \begin{align*}
	 	\left. \frac{1}{h}\frac{d}{ds} J(\tilde\rho^\star + s\nu) \right|_{s=0} =-\log 2<0,
	 \end{align*}
	 which contradicts to the assumption that $\tilde\rho^\star$ is a minimizer, and so the situation \eqref{en:bb} cannot occur.

	 On the other hand, we suppose $\tilde\rho^\star$ touches the other boundary with
	 \begin{align}\label{eq:bb1}
	 	\sum_{i=1}^{n-1}\rho^\star_{i,\ell^0} = 1-\delta.
	 \end{align}
	 Suppose $\rho^\star_{i,\ell^0}$ achieves its maximum at $i_0$, then
	 \begin{align}\label{eq:bb2}
	 	\rho^\star_{i_0,\ell^0} \ge \frac{1-\delta}{n-1} \ge \frac{1}{2(n-1)},
	 \end{align}
	 for $\delta \le \frac{1}{2}$.

	 Since $\tilde\rho^\star-\tilde\rho^k \in ({\mathring{\C}}_{\rm per}^d)^{n-1}$, we have
	 \begin{align*}
	 	\sum_{\ell\in \{1,\ldots,N\}^d}\sum_{i=1}^{n-1} \rho^\star_{i,\ell} = \sum_{\ell\in \{1,\ldots,N\}^d}\sum_{i=1}^{n-1} \rho^k_{i,\ell} \le N^d(1-\rho_{\min}^k)
	 \end{align*}
	 with 
	 \begin{align*}
	 	\rho_{\min}^k = \min_{\substack{i=1,\ldots,n, \\ \ell\in \{1,\ldots,N\}^d}} \rho_{i,\ell}^k.
	 \end{align*}
	 Suppose $\sum_{i=1}^{n-1}\rho^\star_{i,\ell}$ achieves its minimum at $\ell^1$, then we have
	 \begin{align*}
	 	\sum_{i=1}^{n-1}\rho^\star_{i,\ell^1} &\le \frac{1}{N^d-1}( N^d(1-\rho_{\min}^k) - (1-\delta)) 
		\\
		&\le 1 - \frac{N^d\rho_{\min}^k-\delta}{N^d-1}\le 1 - \frac{2N^d-1}{2(N^d-1)}\rho_{\min}^k.
	 \end{align*}
	 if $\delta \le \frac{1}{2}\rho_{\min}^k$.

	 We take 
	 	 \begin{align*}
	 	\nu_{i,\ell}=\left\{\begin{array}{cl}
	 		-1, &\text{for }i=i_0,\; \ell=\ell^0,\\
	 		1, &\text{for }i=i_0,\; \ell=\ell^1,\\
	 		0, &\text{otherwise},
	 	\end{array}\right.
	 \end{align*}
	 and use the above inequality together with \eqref{eq:bb1},\eqref{eq:bb2} to obtain
	 \begin{align*}
	 	&\left. \frac{1}{h}\frac{d}{ds} J(\tilde\rho^\star + s\nu) \right|_{s=0} \nonumber\\
	 	&\quad = -\frac{1}{\Delta t} (\mathcal{L}^{-1}_{\hat{D}^k}(\tilde\rho^\star-\tilde\rho^k))_{i_0,\ell_0}  -  \log \rho^\star_{i_0,\ell_0}+ \log \left(1-\sum_{i=1}^{n-1}\rho^\star_{i,\ell_0}\right)\nonumber\\
	 	&\qquad+\frac{1}{\Delta t} (\mathcal{L}^{-1}_{\hat{D}^k}(\tilde\rho^\star-\tilde\rho^k))_{i_0,\ell_1}  +  \log \rho^\star_{i_0,\ell_1} - \log \left(1-\sum_{i=1}^{n-1}\rho^\star_{i,\ell_1}\right)\nonumber\\
	 	&\quad\le \frac{4C}{\lambda_{\min}^k\Delta t} h^{-\frac{1}{2}}(n-1)^{\frac12} - \log \frac{1}{2(n-1)} + \log \delta 
	 	+ \log 1 - \log \frac{2N^d-1}{2(N^d-1)} \rho_{\min}^k.
	 \end{align*}
	 Taking 
	 \begin{align}\label{eq:dchoice2}
	 	\delta_0 \le \min\left\{ \frac{(2N^d-1)\rho_{\min}^k}{8(N^d-1)(n-1)} e^{-\frac{4C}{\lambda_{\min}^k\Delta t} h^{-\frac{1}{2}}(n-1)^{\frac12}}, 
		                    \frac{1}{2}\rho_{\min}^k,  \frac{1}{4(n-1)}   \right\}
	 \end{align}
	 leads to
	 \[\left. \frac{1}{h^d}\frac{d}{ds} J(\tilde\rho^\star + s\nu) \right|_{s=0}\le -\log 2 <0,\]
	 which contradicts to the assumption that $\tilde\rho^\star$ is a minimizer. 

	 We conclude that there exists a $\delta_0$, which can be chosen to be the smaller value of \eqref{eq:dchoice1}, \eqref{eq:delta0choice3} and \eqref{eq:dchoice2} that only depends on $h,\Delta t,\rho^k$ and the domain, such that the minimizer of \eqref{eq:opd} cannot touch the boundary.

	 To prove the equivalence of the numerical scheme with the minimizer of the optimization problem \eqref{eq:opd}, we follow Step 3 of the proof of Theorem \ref{theorem} for the one dimensional case. We omit the details here.
	 \end{proof}

	 Theorem \ref{thm:propd} is then proved in a fashion similar to the proof of Theorem \ref{thm:prop}.

\section{Proof of consistency}
Here we present detailed calculations of the truncation error defined by 
\begin{align*}
    &\tau_{i}^1 = \frac{P_i^{k+1}-P_i^k}{\Delta t} +d_h(\hat{P}_i^k V_i^{k+1}), \\
	&\tau_i^2 = D_h \log P_i^{k+1} - \frac{1}{\sum_{j=1}^n \hat{P}_j^k}\sum_{i=1}^n \hat{P}_i^k D_h \log P_i^{k+1} +\sum_{j=1}^n b_{ij}  \hat{P}_j^k (V_i^{k+1}-V_j^{k+1}),\\
	&\tau_i^3=\sum_{i=1}^n \hat{P}_i^k V_i^{k+1}.
\end{align*}

We first calculate $\tau^1_i$. 
\begin{align*}
    \tau_{i,\ell}^1 =&\frac{P_{i,\ell}^{k+1}-P_{i,\ell}^k}{\Delta t} +d_h\left(\hat{P}_{i}^k V_{i}^{k+1}\right)_\ell \\
    =& \frac{P_{i,\ell}^{k+1}-P_{i,\ell}^k}{\Delta t} +\frac{1}{h}\left(\hat{P}_{i,\ell+\frac{1}{2}}^k V_{i,\ell+\frac{1}{2}}^{k+1} - \hat{P}_{i,\ell-\frac{1}{2}}^k V_{i,\ell-\frac{1}{2}}^{k+1} \right) , \\
    =& \frac{P_{i,\ell}^{k+1}-P_{i,\ell}^k}{\Delta t} +\frac{1}{2h}\left(({P}_{i,\ell}^k+P_{i,\ell+1}^k) V_{i,\ell+\frac{1}{2}}^{k+1} - ({P}_{i,\ell}^k + P_{i,\ell-1}^k) V_{i,\ell-\frac{1}{2}}^{k+1} \right) .
\end{align*}
The terms in the above equation can be calculated using Taylor's expansion as
\begin{align*}
    P_{i,\ell}^{k+1} =& P_{i,\ell}^k + \partial_t P_{i,\ell}^k \Delta t + O(\Delta t^2),\\
    P_{i,\ell\pm 1}^k =& P_{i,\ell}^k \pm h \partial_x P_{i,\ell}^k + \frac12 h^2 \partial_{xx} P_{i,\ell}^k+O(h^3),\\
    V_{i,\ell\pm \frac12}^{k+1} =& V_{i,\ell}^k \pm \frac{1}{2}h\partial_x V_{i,\ell}^k + \Delta t \partial_t V_{i,\ell}^k + \frac{1}{4}h^2\partial_{xx}V_{i,\ell}^k + \frac{1}{2}\Delta t^2 V_{i,\ell}^k \pm \frac{1}{2}h\Delta t \partial_{xt}V_{i,\ell}^k \\
    &+ O(h^3+\Delta t h^2+\Delta t^2 h + \Delta t^3).
\end{align*}
Taking these expressions into the previous equation leads to 
\begin{align*}
    \tau_{i,\ell}^1 =&  \partial_t P_{i,\ell}^k  
   -\frac{1}{2h} \left(2h\partial_xP_{i,\ell}^k \left(V_{i,\ell}^k  + \Delta t \partial_t V_{i,\ell}^k + \frac{1}{4}h^2\partial_{xx}V_{i,\ell}^k + \frac{1}{2}\Delta t^2 V_{i,\ell}^k\right)\right) \\
    &- \frac{1}{2h} \left(2P_{i,\ell}^k + \frac12 h^2\partial_{xx}P_{i,\ell}^k\right) \left(h\partial_x V_{i,\ell}^k + h\Delta t\partial_{xt}V_{i,\ell}^k\right) \\
    &+ O(\Delta t +h^2+\Delta th +\Delta t^2 +{\Delta t^3}) \\
    =& (\partial_t P - \partial_x(P V))_{i,\ell}^k +  O(\Delta t +h^2+\Delta t^2 + \Delta t h+ {\Delta t^3}).
\end{align*}
The terms $\tau^2$ and $\tau^3$ can be also approximated agin using the Taylor expansion. The results are
\begin{align*}
	    \tau_{i,\ell+\frac12}^2
	     =& \frac{\partial_x P_{i,\ell}^k}{P_{i,\ell}^k}+\sum_{j=1}^n b_{ij} P_{j,\ell}^k (V_{i,\ell}^k-V_{j,\ell}^k) + \frac{h}{2} \bigg[ \frac{\partial_{xx}P_{i,\ell}^k}{P_{i,\ell}^k} - \frac{(\partial_xP_{i,\ell}^k)^2}{(P_{i,\ell}^k)^2} \\
&-\sum_{i=1}^n\bigg(\frac{(\partial_xP_{i,\ell}^k)^2}{P_{i,\ell}^k}+\partial_{xx}P_{i,\ell}^k - \frac{(\partial_xP_{i,\ell}^k)^2}{P_{i,\ell}^k}\bigg) \\
&+\sum_{j=1}^n b_{ij}\left(\partial_xP_{j,\ell}^k (V_{i,\ell}^k-V_{j,\ell}^k) + P_{j,\ell}^k( \partial_xV_{i,\ell}^k-\partial_xV_{j,\ell}^k)\right)\bigg]  + O(\Delta t+ h^2) \\
=& 0 +  \frac{h}{2} \partial_x \left(\frac{\partial_x P_{i,\ell}^k}{P_{i,\ell}^k}-  \sum_{j=1}^n b_{ij} P_{j,\ell}^k (V_{i,\ell}^k-V_{j,\ell}^k)\right) + O(\Delta t + h^2) \\
=& O(\Delta t + h^2).
\end{align*}
\begin{align*}
	 \tau_{i,\ell+\frac12}^3 =&\sum_{i=1}^n P_{i,\ell}^k V_{i,\ell}^k + \frac{1}{2} h \sum_{i=1}^n \partial_x(P_{i,\ell}^k V_{i,\ell}^k) + \Delta t\sum_{i=1}^n P_{i,\ell}^k\partial_tV_{i,\ell}^k + O(\Delta t^2 +h^2) \\
    =& 
    O(\Delta t +h^2) .
\end{align*}
In summary, we conclude the result stated in Lemma \ref{lema}, i.e.,    there exists $C>0$ depending on  $(P,V)$
so that  
\begin{align*}
            |\tau_{i,\ell}^1|,~|\tau_{i,\ell+\frac12}^2|,~|\tau_{i,\ell+\frac12}^3| \le C(\Delta t + h^2).
\end{align*}


\begin{thebibliography}{10}
\bibitem{benamou2000computational}
Jean-David Benamou and Yann Brenier.
\newblock A computational fluid mechanics solution to the {M}onge{-K}antorovich
  mass transfer problem.
\newblock {\em Numerische Mathematik}, 84(3):375--393, 2000.

\bibitem{bothe2011maxwell}
Dieter Bothe.
\newblock On the {M}axwell{-S}tefan approach to multicomponent diffusion.
\newblock In {\em Parabolic problems}, pages 81--93. Springer, 2011.

\bibitem{boudin2012mathematical}
Laurent Boudin, B{\'e}r{\'e}nice Grec, and Francesco Salvarani.
\newblock A mathematical and numerical analysis of the {M}axwell{-S}tefan
  diffusion equations.
\newblock {\em Discrete and Continuous Dynamical Systems-Series B},
  17(5):1427--1440, 2012.

\bibitem{boyd2004convex}
Stephen Boyd and Lieven Vandenberghe.
\newblock {\em Convex optimization}.
\newblock Cambridge university press, 2004.

\bibitem{chen2019positivity}
Wenbin Chen, Cheng Wang, Xiaoming Wang, and Steven~M Wise.
\newblock Positivity-preserving, energy stable numerical schemes for the
  {C}ahn{-H}illiard equation with logarithmic potential.
\newblock {\em Journal of Computational Physics: X}, 3:100031, 2019.

\bibitem{dong2019positivity}
Lixiu Dong, Cheng Wang, Hui Zhang, and Zhengru Zhang.
\newblock A positivity-preserving, energy stable and convergent numerical
  scheme for the {C}ahn{-H}illiard equation with a {F}lory{-H}uggins{-D}egennes
  energy.
\newblock {\em Communications in Mathematical Sciences}, 17(4):921--939, 2019.

\bibitem{geiser2015numerical}
Juergen Geiser.
\newblock Numerical methods of the {M}axwell{-S}tefan diffusion equations and
  applications in plasma and particle transport.
\newblock {\em ArXiv preprint arXiv:1501.05792}, 2015.

\bibitem{giovangigli1998local}
Vincent Giovangigli and Marc Massot.
\newblock The local {C}auchy problem for multicomponent reactive flows in full
  vibrational non-equilibrium.
\newblock {\em Mathematical Methods in the Applied Sciences},
  21(15):1415--1439, 1998.

\bibitem{gong2020arbitrarily}
Yuezheng Gong, Jia Zhao, and Qi~Wang.
\newblock Arbitrarily high-order unconditionally energy stable schemes for
  thermodynamically consistent gradient flow models.
\newblock {\em SIAM Journal on Scientific Computing}, 42(1):B135--B156, 2020.

\bibitem{huo2019high}
Xiaokai Huo, Ansgar J{\"u}ngel, and Athanasios~E Tzavaras.
\newblock High-friction limits of {E}uler flows for multicomponent systems.
\newblock {\em Nonlinearity}, 32(8):2875, 2019.

\bibitem{huo2019positivity}
Xiaokai Huo and Hailiang Liu.
\newblock A positivity-preserving energy stable scheme for a quantum diffusion
  equation.
\newblock {\em ArXiv preprint arXiv:1912.00813}, 2019.

\bibitem{jko98} 
Richard Jordan, David Kinderlehrer, and Felix Otto.
\newblock The variational formulation of the Fokker--Planck equation.
\newblock {\em SIAM Journal on Mathematical Analysis},  29(1), 1--17, 1998.


\bibitem{jungel2016entropy}
Ansgar J{\"u}ngel.
\newblock {\em Entropy methods for diffusive partial differential equations}.
\newblock Springer, 2016.

\bibitem{jungel2019convergence}
Ansgar J{\"u}ngel and Oliver Leingang.
\newblock Convergence of an implicit {E}uler {G}alerkin scheme for
  {P}oisson{-M}axwell{-S}tefan systems.
\newblock {\em Advances in Computational Mathematics}, 45(3):1469--1498, 2019.

\bibitem{jungel2013existence}
Ansgar Jungel and Ines~Viktoria Stelzer.
\newblock Existence analysis of {M}axwell-{S}tefan systems for multicomponent
  mixtures.
\newblock {\em SIAM Journal on Mathematical Analysis}, 45(4):2421--2440, 2013.

\bibitem{krishna1997maxwell}
R~Krishna and JA~Wesselingh.
\newblock The {M}axwell{-S}tefan approach to mass transfer.
\newblock {\em Chemical Engineering Science}, 52(6):861--911, 1997.

\bibitem{liluwang2020}
Wuchen Li, Jianfeng Lu and Li Wang.
\newblock Fisher information regularization schemes for Wasserstein Gradient Flows
(preprint), 2020.

\bibitem{liu2019positive}
Hailiang Liu and Wumaier Maimaitiyiming.
\newblock Positive and free energy satisfying schemes for diffusion with
  interaction potentials.
\newblock {\em ArXiv preprint arXiv:1910.00151}, 2019.

\bibitem{liu2020efficient}
Hailiang Liu and Wumaier Maimaitiyiming.
\newblock Efficient, positive, and energy stable schemes for multi{-D}
  {P}oisson{-N}ernst{-P}lanck systems.
\newblock {\em ArXiv preprint arXiv:2001.08350}, 2020.

\bibitem{otto01}
Felix Otto. 
\newblock The geometry of dissipative evolution equations: The porous medium equation.
\newblock {\em Comm. Partial Differential Equations}, 26 (2001), 101-174.
\newblock 

\bibitem{ottowest06}
Felix Otto and Michael Westdickenberg.
\newblock Eulerian Calculus for the contraction in the Wasserstein distance.
\newblock {\em SIAM J. Math. Anal.} 37 (2006), 1227-1255.

\bibitem{ottotzavaras08}
Felix Otto and Athanasios Tzavaras.
\newblock Continuity of velocity gradients in suspensions of rod-like molecules.
\newblock {\em Comm. Math. Physics} 277 (2008), 729-758.


\bibitem{shen2019new}
Jie Shen, Jie Xu, and Jiang Yang.
\newblock A new class of efficient and robust energy stable schemes for
  gradient flows.
\newblock {\em SIAM Review}, 61(3):474--506, 2019.

\bibitem{wise2009energy}
Steven~M Wise, Cheng Wang, and John~S Lowengrub.
\newblock An energy-stable and convergent finite-difference scheme for the
  phase field crystal equation.
\newblock {\em SIAM Journal on Numerical Analysis}, 47(3):2269--2288, 2009.

\bibitem{yang2015rigorous}
Zaibao Yang, Wen-An Yong, and Yi~Zhu.
\newblock A rigorous derivation of multicomponent diffusion laws.
\newblock {\em ArXiv preprint arXiv:1502.03516}, 2015.

\end{thebibliography}
\end{document}